\documentclass[a4paper,11pt]{article}

\usepackage{graphicx,amsmath,amssymb,amsthm,enumerate,url,color}

\newcommand{\F}[1]{\mathbb{F}_{#1}}
\newcommand{\K}[1]{\mathbb{K}(#1)}
\newcommand{\0}{\mathbf{0}}
\newcommand{\C}{\mathcal{C}}

\newcommand{\GL}[1]{{\rm GL}\left(#1\right)}
\newcommand{\GO}[1]{{\rm GO}\left(#1\right)}
\newcommand{\GOplus}[1]{{\rm GO^+}(#1)}
\newcommand{\GOminus}[1]{{\rm GO^-}(#1)}
\newcommand{\GU}[1]{{\rm GU}\left(#1\right)}
\newcommand{\GSp}[1]{{\rm GSp}\left(#1\right)}
\newcommand{\Sp}[1]{{\rm Sp}\left(#1\right)}
\newcommand{\U}[1]{{\rm U}\left(#1\right)}
\renewcommand{\O}[1]{{\rm O}\left(#1\right)}
\newcommand{\Oplus}[1]{{\rm O^+}\left(#1\right)}
\newcommand{\Ominus}[1]{{\rm O^-}\left(#1\right)}
\newcommand{\SL}[1]{{\rm SL}\left(#1\right)}
\newcommand{\AGL}[1]{{\rm AGL}\left(#1\right)}
\newcommand{\GamL}[1]{{\rm \Gamma L}\left(#1\right)}
\newcommand{\GamSp}[1]{{\rm \Gamma Sp}\left(#1\right)}
\newcommand{\GamU}[1]{{\rm \Gamma U}\left(#1\right)}
\newcommand{\GamO}[1]{{\rm \Gamma O}\left(#1\right)}

\newcommand{\PGamL}[1]{{\rm P \Gamma L}\left(#1\right)}
\newcommand{\GamI}[1]{{\rm \Gamma I}\left(#1\right)}

\newcommand{\Sym}[1]{{\rm Sym}\left(#1\right)}
\newcommand{\Aut}[1]{{\rm Aut}\left(#1\right)}

\newcommand{\Cay}[1]{{\rm Cay}(#1)}
\newcommand{\diam}[1]{{\rm diam}(#1)}

\newcommand{\im}[1]{{\rm Im} \,(#1)}
\newcommand{\V}[1]{V(#1)}
\newcommand{\E}[1]{E(#1)}

\newcommand{\Stab}[2]{{\textnormal{Stab}}_{#1}{\left( #2 \right)}}

\newcommand{\Dim}[2]{{\textnormal{dim}}_{#2}{\left(#1\right)}}
\renewcommand{\dim}[1]{{\textnormal{dim}}\left(#1\right)}
\renewcommand{\min}[1]{{\textnormal{min}}\left\{#1\right\}}

\newcommand{\dist}[2]{{\rm dist}_{#2}{(#1)}}

\newcommand{\wt}[1]{{\rm wt}{\left( #1 \right)}}

\newtheorem{lemma}{Lemma}[subsection]
\newtheorem{proposition}[lemma]{Proposition}
\newtheorem{theorem}[lemma]{Theorem}
\newtheorem{corollary}[lemma]{Corollary}

\newtheorem{maintheorem}{Theorem}[section]
\newtheorem{hypothesis}{Hypothesis}[section]

\theoremstyle{definition}

\newtheorem*{remark}{Remark}
\newtheorem*{notation}{Notation}

\title{Affine primitive symmetric graphs of diameter two
}

\author{Carmen Amarra \footnote{Corresponding author. mcamarra@math.upd.edu.ph \quad Institute of Mathematics, University of the Philippines - Diliman, C. P. Garcia Avenue, Diliman, Quezon City 1101, Philippines}
\and Michael Giudici \footnote{Michael.Giudici@uwa.edu.au \quad Centre for the Mathematics of Symmetry and Computation, The University of Western Australia, 35 Stirling Highway, Perth, WA 6009, Australia}
\and Cheryl E. Praeger \footnote{Cheryl.Praeger@uwa.edu.au \quad Centre for the Mathematics of Symmetry and Computation, The University of Western Australia 35 Stirling Highway, Perth, WA 6009, Australia}
}

\date{\today}

\begin{document}

\maketitle

\begin{abstract}
Let $n$ be a positive integer, $q$ be a prime power, and $V$ be a vector space of dimension $n$ over $\F{q}$. Let $G := V \rtimes G_0$, where $G_0$ is an irreducible subgroup of $\GL{V}$ which is maximal by inclusion with respect to being intransitive on the set of nonzero vectors. We are interested in the class of all diameter two graphs $\Gamma$ that admit such a group $G$ as an arc-transitive, vertex-quasiprimitive subgroup of automorphisms. In particular, we consider those graphs for which $G_0$ is a subgroup of either $\GamL{n,q}$ or $\GamSp{n,q}$ and is maximal in one of the Aschbacher classes $\C_i$, where $i \in \{2,4,5,6,7,8\}$. We are able to determine all graphs $\Gamma$ which arise from $G_0 \leq \GamL{n,q}$ with $i \in \{2,4,8\}$, and from $G_0 \leq \GamSp{n,q}$ with $i \in \{2,8\}$. For the remaining classes we give necessary conditions in order for $\Gamma$ to have diameter two, and in some special subcases determine all $G$-symmetric diameter two graphs.

\end{abstract}

\section{Introduction} \label{sec:intro}

A \emph{symmetric} graph is one which admits a subgroup of automorphisms that acts transitively on its arc set; if $G$ is such a subgroup, we say in particular that the graph is \emph{$G$-symmetric}. We are interested in the family of all symmetric graphs with diameter two, a family which contains all symmetric strongly regular graphs. We consider those $G$-symmetric diameter two graphs where $G$ is a primitive group of affine type, and where the point stabiliser $G_0$ is maximal in the general semilinear group or in the symplectic semisimilarity group. Our main result is Theorem \ref{mainthm:HA}. Those affine examples where $G_0$ is not contained in either of these groups were studied in \cite{HAcase}.

\begin{maintheorem} \label{mainthm:HA}
Let $V = \F{q}^n$ for some prime power $q$ and positive integer $n$, and let $G = V \rtimes G_0$, where $G_0$ is an irreducible subgroup of the general semilinear group $\GamL{n,q}$ or the symplectic semisimilarity group $\GamSp{n,q}$, and $G_0$ is maximal by inclusion with respect to being intransitive on the set of nonzero vectors in $V$. If $\Gamma$ is a connected graph with diameter two which admits $G$ as a symmetric group of automorphisms, then $\Gamma$ is isomorphic to a Cayley graph $\Cay{V,S}$ for some orbit $S$ of $G_0$ satisfying $\langle S \rangle = V$ and $S = -S$, and one of the following holds:
	\begin{enumerate}
	\item $(G_0,S)$ are as in Tables {\rm\ref{tab:mainthm-HA}} and {\rm\ref{tab:mainthm-HA-Sp}};
	\item $G_0$ satisfies the conditions in Table {\rm\ref{tab:mainthm-HA2}};
	\item $G_0$ belongs to the class $\C_9$.
	\end{enumerate}
Furthermore, all pairs $(G_0,S)$ in Tables {\rm\ref{tab:mainthm-HA}} and {\rm\ref{tab:mainthm-HA-Sp}} yield $G$-symmetric diameter two graphs $\Cay{V,S}$.
\end{maintheorem}

\bigskip\noindent
\textbf{Notation for Tables \ref{tab:mainthm-HA} and \ref{tab:mainthm-HA-Sp}.}
The set $X_s$ is as in (\ref{eq:X_s}) and $W_\beta$ is as in (\ref{eq:W_beta}) in Section \ref{sec:C2}, $Y_s$ is as in (\ref{eq:Y_s}) in Section \ref{sec:C4}, $c(v)$ is as in (\ref{eq:c(v)}) in Section \ref{sec:C5}, $S_0$ is as in (\ref{eq:S_lambda}) in Section \ref{subsec:semilin}, and $S_\#$, $S_\square$ and $S_\boxtimes$ are as in (\ref{eq:S_theta}) in Section \ref{sec:C8}. Cayley graphs are defined in Section \ref{subsec:Cay}. The graphs marked $\dagger$ did not appear in \cite{HAcase}.

\begin{table}[ht]
\begin{center}
\renewcommand{\arraystretch}{1.25} \small
\begin{tabular}{rlll}
\hline
& $G_0 \cap \GL{n,q}$ & $S$ & Conditions \\
\hline\hline
{\footnotesize\textsf{1}} & $\GL{m,q} \wr \Sym{t}$, $mt = n$ & $X_s$ & $q^m > 2$ and $s \geq t/2$ \\
{\footnotesize\textsf{2}} & $\GL{k,q} \otimes \GL{m,q}$, $km = n$ & $Y_s$ & $s \geq \frac{1}{2}\min{k,m}$ \\
{$^\dagger$\footnotesize\textsf{3}} & $\GL{n,q^{1/r}} \circ Z_{q-1}$, $r > 2$ and $n > 2$ & $v^{G_0}$ as in (\ref{eq:C5orbits}) & $c(v) = r-1$ or $c(v) = r$ \\
{$^\dagger$\footnotesize\textsf{4}} & $\GL{n,q^{1/r}} \circ Z_{q-1}$, $r=2$ or $n=2$ & $v^{G_0}$ as in (\ref{eq:C5orbits}) &  $c(v) = 1$ \\
{\footnotesize\textsf{5}} & $(Z_{q-1} \circ (Z_4 \circ Q_8)).\Sp{2,2}$, $n = 2$, $q$ odd & $v^{G_0}$ & $v \in V^\#$ \\
{\footnotesize\textsf{6}} & $\GL{m,q} \wr_\otimes \Sym{2}$, $m^2 = n$ & $Y_s$ & $s \geq m/2$ \\
{$^\dagger$\footnotesize\textsf{7}} & $\GU{n,q}$, $n \geq 2$ & $S_0$, $S_\#$ & \\
{\footnotesize\textsf{8}} & $\GO{n,q}$, $n = 3$ and $q = 3$ & $S_0$ \\
{\footnotesize\textsf{9}} & $\GO{n,q}$, $nq$ odd, $n > 3$ or $q > 3$ & $S_0$, $S_\square$, or $S_\boxtimes$ \\
{\footnotesize\textsf{10}} & $\GOplus{n,q}$, $n$ even, $q$ odd, $n > 2$ or $q > 2$ & $S_0$ or $S_\#$ \\
{\footnotesize\textsf{11}} & $\GOminus{n,q}$, $n$ even, $q$ odd, $n > 2$ & $S_0$ or $S_\#$ \\
\hline
\end{tabular} \caption{Symmetric diameter two graphs from maximal subgroups of $\GamL{n,q}$} \label{tab:mainthm-HA}
\end{center}
\end{table}

\begin{table}[ht]
\begin{center}
\renewcommand{\arraystretch}{1.25} \small
\begin{tabular}{rlll}
\hline
& $G_0 \cap \GL{n,q}$ & $S$ & Conditions \\
\hline\hline
{\footnotesize\textsf{1}} & $\Sp{m,q}^t.[q-1].\Sym{t}$, $mt = n$ & $X_s$ & $q^m > 2$ and $s \geq t/2$ \\
{$^\dagger$\footnotesize\textsf{2}} & $\GL{m,q}.[2]$, $2m = n$ & $\bigcup_{\sigma \in \Aut{\F{q}}}W_{\beta^\sigma}$ & $q^m > 2$ and $\beta \in \F{q}$ \\ 
{$^\dagger$\footnotesize\textsf{3}} & $(Z_{q-1} \circ Q_8).\Ominus{2,2}$, $n = 2$, $q$ odd & $v^{G_0}$ & $v \in V^\#$ \\
{\footnotesize\textsf{4}} & $\GOplus{n,q}$, $n = 2$ and $q = 2$ & $S_0$ \\
{\footnotesize\textsf{5}} & $\GOplus{n,q}$, $q$ and $n$ even, $n > 2$ or $q > 2$ & $S_0$ or $S_\#$ \\
{\footnotesize\textsf{6}} & $\GOminus{n,q}$, $q$ and $n$ even, $n > 2$ & $S_0$ or $S_\#$ \\
\hline
\end{tabular} \caption{Symmetric diameter two graphs from maximal subgroups of $\GamSp{n,q}$} \label{tab:mainthm-HA-Sp}
\end{center}
\end{table}

\begin{table}[ht]
\begin{center}
\renewcommand{\arraystretch}{1.25} \small
\begin{tabular}{rlll}
\hline
 & $G_0 \cap \GL{n,q}$ & Conditions & Restrictions \\
\hline\hline
{\footnotesize\textsf{1}} & $\GSp{k,q} \otimes {\rm GO}^\epsilon(m,q)$, $m$ odd, $q > 3$ &  & Proposition \ref{prop:C4-diam2-Sp} 
\\
{\footnotesize\textsf{2}} & $\GL{n,q^{1/r}} \circ Z_{q-1}$ & $c(v) \neq r-1,r$ & Proposition \ref{prop:C5main} (2), (3), (4) 
\\
{\footnotesize\textsf{3}} & $(Z_{q-1} \circ R).\Sp{2t,r}$, $n = r^t$ & $R$ Type 1, $t \geq 2$ & Proposition \ref{prop:C6bounds} (1) 
\\
{\footnotesize\textsf{4}} & $(Z_{q-1} \circ R).\Sp{2t,2}$, $n = r^t$ & $R$ Type 2, $t \geq 2$ & Proposition \ref{prop:C6bounds} (2) 
\\
{\footnotesize\textsf{5}} & $(Z_{q-1} \circ R).\Ominus{2t,2}$, $n = r^t$ & $R$ Type 4, $t \geq 2$ & Proposition \ref{prop:C6bounds} (3) 
\\
{\footnotesize\textsf{6}} & $\GL{m,q} \wr_\otimes \Sym{t}$, $m^t = n$ & $t \geq 3$ & Proposition \ref{prop:C7bounds} 
\\
{\footnotesize\textsf{7}} & $\GSp{m,q} \wr_\otimes \Sym{t}$, $m^t = n$, $q$ odd & $t \geq 3$ & Proposition \ref{prop:C7bounds-Sp} 
\\
\hline
\end{tabular} \caption{Restrictions for remaining cases} \label{tab:mainthm-HA2}
\end{center}
\end{table}

The reduction to these cases is achieved as follows. It is shown in \cite{quotcomp} that any symmetric diameter two graph has a normal quotient graph $\Gamma$ which is $G$-symmetric for some group $G$ and which satisfies one of the following:
	\begin{enumerate}[(I)]
	\item the graph $\Gamma$ has at least one nontrivial $G$-normal quotient, and all nontrivial $G$-normal quotients of $\Gamma$ are complete graphs (that is, every pair of distinct vertices are adjacent); or \label{case:q-c}
	\item all $G$-normal quotients of $\Gamma$ are trivial graphs (that is, consisting of a single vertex). \label{case:qp}
	\end{enumerate}
The context of our investigation is the following. It was shown that those that satisfy (\ref{case:qp}) fall into eight types according to the action of $G$ \cite{Pra93}. One of these types is known as HA (see Subsection \ref{subsec:Cay}). In this case, the vertex set is a finite-dimensional vector space $V = \F{p}^d$ over a prime field $\F{p}$ and $G = V \rtimes G_0$, where $V$ is identified with the group of translations on itself and $G_0$ is an irreducible subgroup of $\GL{d,p}$ which is intransitive on the set of nonzero vectors of $V$. The irreducible subgroups of $\GL{d,p}$ can be divided into eight classes $\C_i$, $i \in \{2, \ldots, 9\}$, most of which can be described as preserving certain geometric configurations on $V$, such as direct sums or tensor decompositions \cite{Asch}. Note that, if a diameter two graph $\Gamma$ is $G$-symmetric, then the stabiliser $G_v$ of a vertex $v$ is not transitive on the remaining vertices since $G_v$ leaves invariant the sets of vertices at distance $1$, and distance $2$, from $v$. Thus, in our situation, the group $G_0$ is intransitive on the set $V^\#$, where $V^\# := V \setminus \{0\}$, the set of nonzero vectors.  In paper \cite{HAcase} we considered the graphs corresponding to the groups $G_0$ which are maximal in their respective classes $\C_i$, for $i \leq 8$, and which are intransitive on nonzero vectors. (We did not consider the last class $\C_9$ since the groups in this class do not have a uniform geometric description.) Several classes were not considered because the maximal groups in these classes are transitive on $V^\#$, namely, the maximal groups are (a) symplectic groups preserving a nondegenerate alternating bilinear form on $V$, and (b) ``extension field groups" preserving a structure on $V$ of an $n$-dimensional vector space over $\F{q}$, where $q^n = p^d$. The aim of this paper is to examine the cases not treated in \cite{HAcase}, namely, $G_0$ preserves either an alternating form or an extension field structure on $V$, and:
	\begin{enumerate}[(I)] \setcounter{enumi}{2}
	\item The group $G_0$ is irreducible and is maximal in $\GL{d,p}$ with respect to being intransitive on nonzero vectors. \label{G_0-condition}
	\end{enumerate}

All quasiprimitive groups of type HA are primitive; the condition of irreducibility of $G_0$ is necessary to guarantee that $G_0$ is maximal in $G$, and hence that $G$ is primitive. In particular, since $G_0$ is intransitive on $V^\#$, $G_0$ does not contain $\SL{V}$ or $\Sp{V}$. The classification in \cite{Asch} can be applied to the two groups $\GamL{n,q}$ and $\GSp{d,p}$: the irreducible subgroups of $\GamL{n,q}$ and of $\GSp{d,p}$ which do not contain $\SL{n,q}$ and $\Sp{d,p}$, respectively, are again organised into classes $\C_2$ to $\C_9$. Again we do not consider the $\C_9$-subgroups. Observe that of the maximal subgroups of $\GamL{n,q}$ in classes $\C_2$ to $\C_8$, the only transitive ones are the $\C_3$-subgroups $\GamL{m,q^{n/m}}$ with $n/m$ prime, and the $\C_8$-subgroup $\GamSp{n,q}$ of symplectic semisimilarities. We avoid these possibilities by choosing $q$ maximal such that $q^n = p^d$. We then consider the two cases: (1) where $G_0 \leq \GamL{n,q}$ and $G_0$ does not  preserve an alternating form on $\F{q}^n$, and (2) where $G_0 \leq \GamSp{n,q}$. Note that in this case it is possible for $d/n$ to be not prime, and it follows from the maximality of $q$ that $G_0$ is not contained in a proper $\C_3$-subgroup of $\GamL{n,q}$ or $\GamSp{n,q}$, respectively. Since $G_0$ is irreducible and we are not considering $\C_9$-subgroups, we now have $G_0$ a maximal intransitive subgroup in the $\C_i$ (for $\GamL{n,q}$ or $\GamSp{n,q}$) for some $i \in \{2, 4, 5, 6, 7, 8\}$.

%

All such subgroups of $\GamL{n,q}$ for which $n = d$ and $i \neq 5$ are considered in \cite{HAcase}; moreover, for some of these cases, the arguments were given in the general setting of $\C_i$-subgroups of $\GamL{n,q}$, and so can be applied here. The cases requiring the most detailed arguments are those for subfield groups and, to a lesser extent, normalisers of symplectic-type $r$-groups ($\C_i$-groups with $i \in \{5, 6\}$).

As in \cite{HAcase}, for each family of groups $G_0$ we have two main tasks:
	\begin{itemize}
	\item to determine the $G_0$-orbits, and
	\item to identify which of these orbits correspond to diameter two Cayley graphs.
	\end{itemize}
In the instances where we are not able to achieve either of these, we obtain bounds on certain parameters to reduce the number of unresolved cases.

The rest of this paper is organised as follows: In Section 2 we give the relevant background on affine quasiprimitive permutation groups, semilinear transformations and semisimilarities. In Subsection 2.3 we present Aschbacher's classification of the subgroups of $\GamL{n,q}$ and $\GamSp{n,q}$. Section 3 is devoted to the proof of Theorem \ref{mainthm:HA}, which we do by considering separately the maximal intransitive subgroups in each of the classes $\C_i$, where $i \in \{ 2, 4, 5, 6, 7, 8 \}$.

\begin{notation}
If $A$ is a vector space, a finite field, or a group, $A^\#$ denotes the set of nonzero vectors, nonzero field elements, or non-identity group elements, respectively. The finite field of order $q$ is denoted by $\F{q}$. The notation used for the classical groups, some of which is nonstandard, is presented in Section \ref{sec:prelims}. If $\Gamma$ is a graph, $\V{\Gamma}$ and $\E{\Gamma}$ are, respectively, its vertex set and edge set.
\end{notation}

\section{Preliminaries} \label{sec:prelims}

\subsection{Cayley graphs and HA-type groups} \label{subsec:Cay}

The action of a group $G$ on a set $\Omega$ is said to be \emph{quasiprimitive of type HA} if $G$ has a unique minimal normal subgroup $N$ and $N$ is elementary abelian and acts regularly on $\Omega$. The group $G$ is then a subgroup of the holomorph $N.\Aut{N}$ of $N$ (hence the abbreviation HA, for \emph{h}olomorph of an \emph{a}belian group). It follows from \cite[Lemma 16.3]{biggs} that a graph $\Gamma$ that admits $G$ as a subgroup of automorphisms is isomorphic to a \emph{Cayley graph} on $N$, that is, a graph with vertex set $N$ and edge set $\{ \{x,y\} \ | \ x-y \in S \}$ for some subset $S$ of $N^\#$ with $S = -S$ and $0 \notin S$. (Since $N$ is abelian we use additive notation, and in particular denote the identity by $0$ and call it zero.) Such a graph is denoted by $\Cay{N,S}$. If, in addition, $\Gamma$ is $G$-symmetric, then $S$ must be an orbit of the point stabiliser $G_0$ of zero. Thus, in order for $\Gamma$ to have diameter two, the group $G_0$ must be intransitive on the set of nonzero elements in $N$.

The result that is most relevant to our investigation is Lemma \ref{lem:HAgraphs}, which follows from the basic properties of Cayley graphs and quasiprimitive groups of type HA.

\begin{lemma} \cite{Pra93} \label{lem:HAgraphs}
Let $\Gamma$ be a graph and $G \leq \Aut{\Gamma}$, where $G$ acts quasiprimitively on $\V{\Gamma}$ and is of type HA. Then $G \cong \F{p}^d \rtimes G_0 \leq \AGL{d,p}$ and $\Gamma \cong \Cay{\F{p}^d,S}$ for some finite field $\F{p}$, where the vector space $\F{p}^d$ is identified with its translation group and $G_0 \leq \GL{d,p}$ is irreducible. Moreover, $\Gamma$ is $G$-symmetric with diameter $2$ if and only if $S$ is a $G_0$-orbit of nonzero vectors satisfying $-S = S$, $S \subsetneq V$ and $S \cup (S+S) = V$.
\end{lemma}

The condition $-S = S$ implies that $|S+S| \leq |S|(|S|-1) + 1$, and if $S$ is a $G_0$-orbit then clearly $|S| \leq |G_0|$. It follows from Lemma \ref{lem:HAgraphs} that if $\Cay{V,S}$ is $G$-symmetric with diameter two then
	\begin{equation} 
	|V| \leq |S|^2 + 1 \leq |G_0|^2 + 1.
	\end{equation}
This fact will be frequently used in obtaining bounds for certain parameters.

In our situation $p^d = q^n$ and $G_0$ preserves on $V$ the structure of an $\F{q}$-space; we therefore regard $V$ as $V = \F{q}^n$, and $G_0$ as a subgroup of $\GamL{n,q}$.

\subsection{Semilinear transformations and semisimilarities} \label{subsec:semilin}

Throughout this subsection assume that $q$ is an arbitrary prime power, $V$ is a vector space with finite dimension $n$ over $\F{q}$, and $\mathcal{B} := \{v_1, \ldots, v_n\}$ is a fixed $\F{q}$-basis of $V$.

The \emph{general semilinear group} $\GamL{n,q}$ consists of all invertible maps $h : V \rightarrow V$ for which there exists $\alpha(h) \in \F{q}$, which depends only on $h$, satisfying
	\begin{equation} \label{eq:semilintransf}
	(\lambda u + v)^h = \lambda^{\alpha(h)} u^h + v^h \quad \text{for all} \ \lambda \in \F{q} \ \text{and} \ u,v \in V.
	\end{equation}
The group $\GamL{n,q}$ is isomorphic to a semidirect product $\GL{n,q} \rtimes \Aut{\F{q}}$ with the following action on $V$:
	\begin{align} \label{eq:GamL-action}
	\left( \sum_{i=1}^n{\lambda_i v_i} \right)^{g\alpha} := \sum_{i=1}^n{\lambda_i^\alpha v_i^g} \quad &\text{for all} \ g \in \GL{n,q}, \ \alpha \in \Aut{\F{q}}, \ \text{and} \\
	&\lambda_1, \ldots, \lambda_n \in \F{q}. \notag
	\end{align}
If $V$ is endowed with a left-linear or quadratic form $\phi$, then the elements of $\GamL{n,q}$ that preserve $\phi$ up to a nonzero scalar factor or an $\F{q}$-automorphism are called \emph{semisimilarities} of $\phi$. That is, $h$ is a semisimilarity of $\phi$ if and only if for some $\lambda(h) \in \F{q}^\#$ and some $\alpha'(h) \in \Aut{\F{q}}$, both of which depend only on $h$,
		\[ \phi(u^h,v^h) = \lambda(h) \phi(u,v)^{\alpha'(h)} \quad \text{for all} \ u,v \in V \]
if $\phi$ is left-linear, and
	\[ \phi(v^h) = \lambda(h) \phi(v)^{\alpha'(h)} \quad \text{for all} \ v \in V \]
if $\phi$ is quadratic. It can be shown that $\alpha'(h)$ is the element $\alpha(h)$ in (\ref{eq:semilintransf}). The set of all semisimilarities of $\phi$ is a subgroup of $\GamL{n,q}$ and is denoted by $\GamI{n,q}$, where $\rm I$ is $\rm Sp, U, O, O^+$, or $\rm O^-$, if $\phi$ is symplectic (i.e., nondegenerate alternating bilinear), unitary (i.e., nondegenerate conjugate-symmetric sesquilinear), quadratic in odd dimension, quadratic of plus type, or quadratic of minus type, respectively.

The map $\alpha : \GamI{n,q} \rightarrow \Aut{\F{q}}$ defined by $h \mapsto \alpha(h)$ is a group homomorphism whose kernel ${\rm GI}(n,q)$ consists of all $g \in \GL{n,q}$ that preserve $\phi$ up to a nonzero scalar factor. The elements of ${\rm GI}(n,q)$ are called \emph{similarities} of $\phi$. Likewise, the map $g \mapsto \lambda(g)$ for any $g \in \GamI{n,q}$ defines a homomorphism $\lambda$ from ${\rm GI}(n,q)$ to the multiplicative group $\F{q}^\#$. The kernel ${\rm I}(n,q)$ of $\lambda$ consists of all $\phi$-preserving elements in $\GL{n,q}$, which are called the \emph{isometries} of $\phi$. It should be emphasised that our notation for the similarity and isometry groups is non-standard, but follows for example \cite{FNP}: the symbol ${\rm GI}(n,q)$ is sometimes used to denote the isometry group, whereas in the present paper this refers to the similarity group.

In Subsection \ref{sec:C8} we determine the orbits in $V^\#$ of the groups $\GamI{n,q}$. The following result, which gives the orbits of the isometry groups ${\rm I}(n,q)$, is useful:

\begin{theorem} \cite[Propositions 3.11, 5.12, 6.8 and 7.10]{Wan} \label{thm:isometry-orbits}
Let $V = \F{q}^n$ and $\phi$ a symplectic, unitary, or nondegenerate quadratic form on $V$. Then the orbits in $V^\#$ of the isometry group of $(V,\phi)$ are the sets $S_\lambda$ for each $\lambda \in \im{\overline{\phi}}$, where
	\begin{equation} \label{eq:S_lambda}
	S_\lambda := \{ v \in V^\# \ | \ \overline{\phi}(v) = \lambda \}
	\end{equation}
and
	\begin{equation} \label{eq:phi-bar}
	\overline{\phi}(v) = \begin{cases}
											 \phi(v,v) \ &\text{if $\phi$ is symplectic or unitary}; \\
											 \phi(v) &\text{if $\phi$ is quadratic}.
											 \end{cases}
	\end{equation}
\end{theorem}

Observe that if $\phi$ is symplectic then $\phi(v,v) = 0$ for all nonzero vectors $v$, so it follows from Theorem \ref{thm:isometry-orbits} that $\Sp{n,q}$ is transitive on $V^\#$.

\subsubsection{Some geometry}

Let $f$ be a left-linear form on $V$. A nonzero vector $v$ is called \emph{isotropic} if $f(v,v) = 0$; otherwise, it is \emph{anisotropic}. If $f$ is symplectic or unitary, then an isotropic vector is also called \emph{singular}. If $f$ is symmetric bilinear and $Q$ is a quadratic form which polarises to $f$ (that is, $f(u,v) = Q(u+v) - Q(u) - Q(v)$), then a singular vector is a nonzero vector $v$ with $Q(v) = 0$. Hence, in general, all isotropic vectors are singular and vice versa, unless $V$ is orthogonal and $q$ is even; in this case all nonzero vectors are isotropic but not all are singular. A subspace $U$ of $V$ is \emph{totally isotropic} if $f \vert_U \equiv 0$, and \emph{totally singular} if all its nonzero vectors are singular. On the other hand, a subspace $U$ is \emph{anisotropic} if all of its nonzero vectors are anisotropic.

For any subspace $U$ of $V$ we define the subspace
	\[ U^\perp := \left\{ v \in V \ \vert \ f(u,v) = 0 \ \forall \ u \in U \right\} \]
and we write $V = U \perp W$ if $V = U \oplus W$ and $W \leq U^\perp$. Clearly a nonzero vector $v$ is isotropic if and only if $v \in \langle v \rangle^\perp$, and the subspace $U$ is totally isotropic if and only if $U \leq U^\perp$. A symplectic or unitary form $f$, or a quadratic form with associated bilinear form $f$, is \emph{nondegenerate} (or \emph{nonsingular}) if the radical $V^\perp$ of $f$ is the zero subspace.

A \emph{hyperbolic pair} in $V$ is a pair $\{x,y\}$ of singular vectors such that $f(x,y) = 1$. The space $V$ can be decomposed into an orthogonal direct sum of an anisotropic subspace and subspaces spanned by hyperbolic pairs, as stated in the following fundamental result on the geometry of formed spaces.

\begin{theorem} \cite[Propositions 2.3.2, 2.4.1, 2.5.3]{KleidLieb} \label{thm:Wittindex}
Let $V = \F{q}^n$, and let $f$ be a left-linear form on $V$ which is symplectic, unitary, or a symmetric bilinear form associated with a nondegenerate quadratic form $Q$. Then
	\[ V = \langle x_1,y_1 \rangle \perp \ldots \perp \langle x_m,y_m \rangle \perp U \]
where $\{ x_i,y_i \}$ is a hyperbolic pair for each $i$ and $U$ is an anisotropic subspace. Moreover:
	\begin{enumerate}
	\item If $f$ is symplectic then $U = 0$. Hence $n$ is even and, up to equivalence, there is a unique symplectic geometry in dimension $n$ over $\F{q}$.
	\item If $f$ is unitary then $U = 0$ if $n$ is even and $\dim{U} = 1$ if $n$ is odd. Hence up to equivalence, there is a unique unitary geometry in dimension $n$ over $\F{q}$.
	\item If $f$ is symmetric bilinear with quadratic form $Q$ and $n$ is odd, then $q$ is odd, $\dim{U} = 1$, and there are two isometry classes of quadratic forms in dimension $n$ over $\F{q}$, one a non-square multiple of the other. Hence all orthogonal geometries in dimension $n$ over $\F{q}$ are similar.
	\item If $f$ is symmetric bilinear with quadratic form $Q$ and $n$ is even, then $U = 0$ or $\dim{U} = 2$. For each $n$ there are exactly two isometry classes of orthogonal geometries over $\F{q}$, which are distinguished by $\dim{U}$.
	\end{enumerate}
\end{theorem}

In Theorem \ref{thm:Wittindex} (4), the quadratic form $Q$ and the corresponding geometry is said to be of \emph{plus type} if $U = 0$, and of \emph{minus type} if $\dim{U} = 2$.

\subsubsection{Tensor products}

Some of the subgroups listed in Aschbacher's classification arise as tensor products of classical groups. In order to describe the group action we define first the tensor product of forms. If $V = U \otimes W$, and if $\phi_U$ and $\phi_W$ are both bilinear or both unitary forms on $U$ and $W$, respectively, then the form $\phi_U \otimes \phi_W$ on $V$ is defined by
	\[ \left( \phi_U \otimes \phi_W \right) \left( u \otimes w, u' \otimes w' \right) := \phi_U(u,u') \phi_W(w,w') \]
for all $u \otimes w$ and $u' \otimes w'$ in a tensor product basis of $V$, extended bilinearly if $\phi_U$ and $\phi_W$ are bilinear, and sesquilinearly if $\phi_U$ and $\phi_W$ are sesquilinear. If $\phi_U$ and $\phi_W$ are both bilinear then so is $\phi_U \otimes \phi_W$; moreover, $\phi_U \otimes \phi_W$ is alternating if at least one of $\phi_U$ and $\phi_W$ is alternating, and $\phi_U \otimes \phi_W$ is symmetric if both $\phi_U$ and $\phi_W$ are symmetric. If $\phi_U$ and $\phi_W$ are both unitary then $\phi_U \otimes \phi_W$ is unitary. The tensor product $I(U,\phi_U) \otimes I(W,\phi_W)$ acts on $V$ with the usual tensor product action --- that is, for any $g \in I(U,\phi_U)$, $h \in I(W,\phi_W)$, $u \in U$ and $w \in W$,
	\[ (u \otimes w)^{(g,h)} := u^g \otimes w^h. \]
The types of forms $\phi_U \otimes \phi_W$ that arise according to the various possibilities for $\phi_U$ and $\phi_W$, which are given in terms of the possible inclusions $I(U,\phi_U) \otimes I(W,\phi_W) \leq I(V,\phi_U \otimes \phi_W)$, are summarised in Table \ref{table:tensor-classical}.

\begin{table}
\begin{center}
\begin{tabular}{lll}
\hline
\multicolumn{1}{c}{${\rm I}(U,\phi_U)$} & \multicolumn{1}{c}{${\rm I}(W,\phi_W)$} & \multicolumn{1}{c}{${\rm I}(U \otimes W,\phi_U \otimes \phi_W)$} \\
\hline\hline
\\[-9pt]
$\rm Sp$ & $\rm O^\epsilon$ & $\begin{cases} \rm Sp \ &\text{if the characteristic is odd;} \\ \rm O^+ &\text{else} \end{cases}$ \\
\\[-9pt]
$\rm Sp$ & $\rm Sp$ & $\rm O^+$ \\
\\[-9pt]
$\rm O^{\epsilon_1}$ & $\rm O^{\epsilon_2}$  & $\begin{cases} \rm O^+ \ &\text{if $\epsilon_i = +$ for some $i$, or $\epsilon_i = -$ for both $i$;} \\ \rm O &\text{if $\dim{U}$ and $\dim{W}$ are odd;} \\ \rm O^- &\text{else} \end{cases}$ \\
\\[-9pt]
$\rm U$ & $\rm U$ & $\rm U$ \\
\hline
\end{tabular} \bigskip \caption{Tensor products of classical groups} \label{table:tensor-classical}
\end{center}
\end{table}

The tensor product of an arbitrary number of formed spaces can be defined similarly: If $V = U_1 \otimes \cdots \otimes U_t$ and $\phi_i$ is a nondegenerate form on $U_i$ for each $i$, and either all $\phi_i$ are bilinear or all are sesquilinear, the form $\phi_1 \otimes \cdots \otimes \phi_t$ is given by
	\[ \left( \otimes_{i=1}^t \phi_i \right) \left( \otimes_{i=1}^t u_i, \otimes_{i=1}^t w_i \right) = \prod_{i=1}^t \phi(u_i, w_i) \]
as $\otimes_{i=1}^t u_i$ and $\otimes_{i=1}^t w_i$ vary over a tensor product basis of $V$, extended bilinearly if the $\phi$ are bilinear, and sesquilinearly if they are sesquilinear. Then $\otimes_{i=1}^t \phi_i$ is a nondegenerate bilinear (respectively, sesquilinear) form on $V$. If the spaces $(U_i,\phi_i)$ are all isometric, then we can extend the results of Table \ref{table:tensor-classical} to the following (see \cite{KleidLieb,Wilson}):
	\begin{align*}
	\otimes_{i=1}^t{\Sp{m,q}} &< \begin{cases} \Sp{m^t,q} \ &\text{if $qt$ odd}; \\ \Oplus{m^t,q} &\text{if $qt$ is even} \end{cases} \\
	\otimes_{i=1}^t{{\rm O}^\epsilon(m,q)} &< \begin{cases} \O{m^t,q} \ &\text{if $qm$ is odd}; \\ \Ominus{m^t,q} &\text{if $\epsilon = -$ and $t$ is odd}; \\ \Oplus{m^t,q} &\text{else} \end{cases} \\
	\otimes_{i=1}^t{\U{m,q}} &< \U{m^t,q}
	\end{align*}

\subsection{Aschbacher's classification} \label{subsec:Asch}

The irreducible subgroups of semisimilarity and semilinear groups are classified by Aschbacher's Theorem \cite{Asch}. In \cite{KleidLieb}, Aschbacher's Theorem is used to identify those irreducible subgroups which are maximal. We present below the versions that correspond to $\GamL{n,q}$ and to $\GamSp{n,q}$. Recall that $G_0$ does not contain either of the transitive groups $\SL{n,q}$ or $\Sp{n,q}$.

\begin{theorem} \label{thm:Asch-GL}
If $M$ is a maximal irreducible subgroup of $\GamL{n,q}$ that does not contain $\SL{n,q}$, then $M$ is one of the following groups:
	\begin{enumerate}
	\item[$(\C_2)$] $\left(\GL{m,q} \wr \Sym{t}\right) \rtimes \Aut{\F{q}}$, where $mt = n$;
	\item[$(\C_3)$] $\GamL{m,q^r}$, where $r$ is prime and $mr = n$; 
	\item[$(\C_4)$] $(\GL{k,q} \otimes \GL{m,q}) \rtimes \Aut{\F{q}}$, where $km = n$ and $k \neq m$, and the action of $\tau$ is defined with respect to a tensor product basis of $\F{q}^k \otimes \F{q}^m$;
	\item[$(\C_5)$] $\left(\GL{n,q^{1/r}} \circ Z_{q-1}\right) \rtimes \Aut{\F{q}}$, where $n \geq 2$, $q$ is an $r$th power and $r$ is prime;
	\item[$(\C_6)$] $\left((Z_{q-1} \circ R).T\right) \rtimes \Aut{\F{q}}$, where $n = r^t$ with $r$ prime, $q$ is the smallest power of $p$ such that $q \equiv 1\pmod{r}$, and $R$ and $T$ are as given in Table {\rm\ref{table:C6groups}} with $R$ of type 1 or 2;
	\item[$(\C_7)$] $\left(\GL{m,q} \wr_\otimes \Sym{t}\right) \rtimes \Aut{\F{q}}$, where $m^t = n$, $t \geq 2$, and the action of $\tau$ is defined with respect to a tensor product basis of $\otimes_{i=1}^t{\F{q}^m}$;
	\item[$(\C_8)$] $\GamO{n,q}$ or ${\rm \Gamma O}^\pm(n,q)$ with $q$ odd, $\GamSp{n,q}$, or $\GamU{n,q}$;
	\item[$(\C_9)$] the preimage of an almost simple group $H \leq \PGamL{n,q}$ satisfying the following conditions:
		\begin{enumerate}
		\item $T \leq H \leq \Aut{T}$ for some nonabelian simple group $T$ (i.e., $H$ is almost simple).
		\item The preimage of $T$ in $\GL{n,q}$ is absolutely irreducible and cannot be realised over a proper subfield of $\F{q}$.
		\end{enumerate}
	\end{enumerate}
\end{theorem}

In Theorem \ref{thm:Asch-Sp} the symbol $[o]$ denotes a group of order $o$. In case $(\C_2)$ the group $[q-1]$ is generated by the map
	\[ \delta_\mu : x_i \mapsto \mu x_i, \ y_i \mapsto y_i \]
for all $x_i$ and all $y_i$, $i \in \{1, \ldots, n/2\}$, where $\mu$ is a generator of the multiplicative group $\F{q}^\#$ and $\{x_1, \ldots, x_{n/2}, y_1, \ldots, y_{n/2}\}$ is a basis of $\F{q}^n$, satisfying $\phi(x_i,x_j) = \phi(y_i,y_j) = \phi(x_i,y_j) = 0$ whenever $i \neq j$ and $\phi(x_i,y_i) = 1$ for all $i$. Such a basis is called a \emph{symplectic basis}.

\begin{theorem} \label{thm:Asch-Sp}
If $M$ is a maximal irreducible subgroup of $\GamSp{n,q}$, then $M$ is one of the following groups:
	\begin{enumerate}
	\item[$(\C_2)$] $\left( (\Sp{m,q}^t.[q-1].\Sym{t}) \right) \rtimes \Aut{\F{q}}$, where $m = n/t$; or \\
		$\left( \GL{m,q}.[2] \right) \rtimes \Aut{\F{q}}$, where $m = n/2$;
	\item[$(\C_3)$] $\left( \Sp{m,q^r}.[q-1] \right) \rtimes \Aut{\F{q}}$, where $r$ is prime and $m = n/r$; or \\
		$\GamU{m,q^2}$, where $m = n/2$ and $q$ is odd;
	\item[$(\C_4)$] $\left( \GSp{k,q} \times {\rm GO}^\epsilon(m,q) \right) \rtimes \Aut{\F{q}}$, where $q$ is odd, $k \neq m$, $m \geq 3$, and ${\rm GO}^\epsilon$ can be any of $\rm GO$, $\rm GO^+$, or $\rm GO^-$;
	\item[$(\C_5)$] $\left( \GSp{n,q^{1/r}} \circ Z_{q-1} \right) \rtimes \Aut{\F{q}}$
	\item[$(\C_6)$] $\left( Z_{q-1} \circ R \right).\Ominus{2t,2}$, where $q \geq 3$ and is prime, and $R$ is of type 4 in Table {\rm\ref{table:C6groups}};
	\item[$(\C_7)$] $\left( \GSp{m,q} \wr_\otimes \Sym{t} \right) \rtimes \Aut{\F{q}}$, where $qt$ is odd;
	\item[$(\C_8)$] ${\rm \Gamma O}^\pm (n,q)$, where $q$ is even;
	\item[$(\C_9)$] the preimage of an almost simple group $H \leq \PGamL{n,q}$ satisfying the following conditions:
		\begin{enumerate}
		\item $T \leq H \leq \Aut{T}$ for some nonabelian simple group $T$ (i.e., $H$ is almost simple).
		\item The preimage of $T$ in $\GL{n,q}$ is symplectic, absolutely irreducible, and cannot be realised over a proper subfield of $\F{q}$.
		\end{enumerate}
	\end{enumerate}
\end{theorem}

\begin{table}[ht]
\renewcommand{\arraystretch}{1.25} \small
\begin{center}
\begin{tabular}{lcll}
\hline  & $r$ & \multicolumn{1}{c}{$R$} & \multicolumn{1}{c}{$T$} \\
\hline\hline
Type 1 & odd & $\underbrace{R_0 \circ \cdots \circ R_0}_t$, $R_0 :=r_+^{1+2}$ & $\Sp{2t,r}$ \\
Type 2 & $2$ & $Z_4 \circ \underbrace{Q_8 \circ \cdots \circ Q_8}_{t}$ & $\Sp{2t,2}$ \\
Type 4 & $2$ & $\underbrace{D_8 \circ \cdots \circ D_8}_{t-1} \circ Q_8$ & $\Ominus{2t,2}$ \\
\hline
\end{tabular} \bigskip \caption{$\C_6$-subgroups} \label{table:C6groups}
\end{center}
\end{table}
\section{Symmetric diameter two graphs from maximal subgroups of $\GamL{n,q}$ and $\GamSp{n,q}$}

In this section we prove Theorem \ref{mainthm:HA}. In view of the observations in Section \ref{sec:intro}, assume that the following hypothesis holds:

\begin{hypothesis} \label{hyp:HA}
Let $V = \F{p}^d$ with $p$ prime and $d \geq 2$, which is viewed as $\F{q}^n$ with $q = p^{d/n}$ for some divisor $n$ of $d$ (possibly $d/n$ composite or $n = d$). Let $H$ be one of the subgroups below of $\GL{d,p}$:
	\begin{enumerate}
	\item $H = \GamL{n,q} = \GL{n,q} \rtimes \langle \tau \rangle$, the general semilinear group on $V$, or
	\item $H = \GamSp{n,q} = \GSp{n,q} \rtimes \langle \tau \rangle$, the group of symplectic semisimilarities of a symplectic form on $V$,
	\end{enumerate}
Let $\tau$ denote the Frobenius automorphism of $\F{q}$ and $\mathcal{B}$ be a fixed $\F{q}$-basis of $V$, with $\tau$ acting on $V$ as in (\ref{eq:GamL-action}) with respect to $\mathcal{B}$ (with $g = 1$ and $\alpha = \tau$); for the case where $H = \GamSp{n,q}$ assume that $\mathcal{B}$ is a symplectic basis of $V$. Define $G = V \rtimes G_0 \leq V \rtimes H < \AGL{d,p}$ and $L = G_0 \cap \GL{n,q}$, where $G_0$ is a maximal $\C_i$-subgroup of $H$ for some $i \in \{2, 4, 5, 6, 7, 8\}$ and $G_0$ does not contain $\Sp{n,q}$ or $\SL{n,q}$.
\end{hypothesis}

We note that the groups considered in \cite{HAcase} are the same as the subgroups $L$, as defined above, of $H = \GamL{n,q}$.

All irreducible subgroups of $\GL{d,p}$ which are maximal with respect to being intransitive on $V^\#$ thus occur as subcases of the groups considered in Hypothesis \ref{hyp:HA} or belong to class $\C_9$. (Indeed, $G_0$ is maximal intransitive if $n = d$ or if $d/n$ is prime.) For each Aschbacher class assume that $G_0 = M$ is of the form given in Theorem \ref{thm:Asch-GL} or \ref{thm:Asch-Sp}.

Since some of the other subgroups of $\GamSp{n,q}$ involve classical groups, we begin with class $\C_8$.

\subsection{Class $\C_8$} \label{sec:C8}

In this case the space $V$ has a form $\phi$, which is symplectic, unitary, or nondegenerate quadratic if $H = \GamL{n,q}$, and is nondegenerate quadratic if $H = \GamSp{n,q}$ with $q$ even. Since the symplectic group is transitive on $V^\#$, we consider only the unitary and orthogonal cases.

Throughout this section we shall use the following notation: for $\theta \in \{\square, \boxtimes, \#\}$ let
	\begin{equation} \label{eq:S_theta}
	S_\theta := \bigcup_{\lambda \in \F{q}^\theta}{S_\lambda}
	\end{equation}
where the $S_\lambda$ are as in (\ref{eq:S_lambda}). If $q$ is a square (as in the unitary case), let $q_0 := \sqrt{q}$ and let $\F{q_0}$ denote the subfield of $\F{q}$ of index 2. Also let ${\rm Tr} : \F{q} \rightarrow \F{q_0}$ denote the trace map, that is, ${\rm Tr}(\alpha) = \alpha + \alpha^{q_0}$ for all $\alpha \in \F{q}$.

We begin by describing the orbits of the similarity groups ${\rm GI}(n,q)$, where $\rm I \in \left\{ U, O,\right.$ $\left. \rm O^+, O^- \right\}$.

\begin{proposition}\label{prop:similarity-orbits}
Let $V = \F{q}^n$, $\phi$ be a unitary or nondegenerate quadratic form on $V$, and $G_0 = {\rm GI}(n,q)$ with $\rm I \in \{U, O, O^+, O^- \}$, according to the type of $\phi$. Let $S_0$ be as in {\rm(\ref{eq:S_lambda})} and $S_\square$, $S_\boxtimes$ and $S_\#$ be as in {\rm(\ref{eq:S_theta})}.
	\begin{enumerate}
	\item If $\phi$ is unitary, then the $G_0$-orbits in $V^\#$ are $S_0$ and $S_\#$.
	\item If $\phi$ is nondegenerate quadratic, then the $G_0$-orbits in $V^\#$ are as follows:
		\begin{enumerate}[(i)]
		\item $S_\#$ if $n = 1$;
		\item $S_0$ and $S_\#$ if $n$ is even;
		\item $S_0$, $S_\square$ and $S_\boxtimes$ if $n$ is odd and $n \geq 3$.
		\end{enumerate}
	\end{enumerate}
\end{proposition}

\begin{proof}
Statement 2 is precisely \cite[Proposition 3.9]{HAcase}, so we only need to prove statement 1. Assume that $\phi$ is unitary; hence $q$ is a square and $q_0 = \sqrt{q}$. It follows from Theorem \ref{thm:isometry-orbits} that $S_0$ is a $G_0$-orbit (that is, provided that $S_0 \neq \varnothing$), so we only need to show that $S_\#$ is a $G_0$-orbit. Let $v \in S_\#$; clearly, $v^{G_0} \subseteq S_\#$. For any $u \in S_\#$ set $\alpha := f(u,u)f(v,v)^{-1}$. Then $\alpha \in \F{q_0}^\#$, so $\alpha = \beta^{q_0 + 1}$ for some $\beta \in \F{q}$. Hence $f(u,u) = \beta^{q_0 + 1} f(v,v) = f(\beta v, \beta v)$, so by Theorem \ref{thm:isometry-orbits} we have $u = (\beta v)^g$ for some $g \in \U{n,q}$. Then $u = v^{\beta g}$, where $\beta g \in \GU{n,q}$. Therefore $v^{G_0} = S_\#$, which proves statement 1.
\end{proof}

The orbits of the semisimilarity groups can be easily deduced from Proposition \ref{prop:similarity-orbits}.

\begin{proposition} \label{prop:semisimilarity-orbits}
Let $V = \F{q}^n$, $\phi$ be a unitary or nondegenerate quadratic form on $V$, and $G_0 = {\rm \Gamma I}(n,q)$ with $\rm I \in \{U, O, O^+, O^- \}$, according to the type of $\phi$. Then for all cases, the $G_0$-orbits are exactly the same as the ${\rm GI}(n,q)$-orbits.
\end{proposition}

\begin{proof}
This follows from Proposition \ref{prop:similarity-orbits} and the fact that the elements of ${\rm \Gamma I}(n,q)$ preserve the form up to an automorphism of $\F{q}$.
\end{proof}

Hence, a direct consequence of Proposition \ref{prop:semisimilarity-orbits} and \cite[Proposition 3.12]{HAcase} is:

\begin{proposition} \label{prop:C8-GOdiam2}
Let $\Gamma$ be a graph and $G \leq \Aut{\Gamma}$ such that $G$ satisfies Hypothesis {\rm\ref{hyp:HA}} with $G_0 = \GamO{n,q}$ or $G_0 = {\rm \Gamma O}^\epsilon(n,q)$ ($\epsilon = \pm$). Then $\Gamma$ is $G$-symmetric with diameter $2$ if and only if $\Gamma \cong \Cay{V,S}$ with $V = \F{q}^n$ and the conditions listed in one of the lines 8--11 of Table {\rm\ref{tab:mainthm-HA}} or lines 4--6 of Table {\rm\ref{tab:mainthm-HA-Sp}} hold.
\end{proposition}

We now consider the unitary case. Note that Theorem \ref{thm:Wittindex} implies that the space $V$ contains a hyperbolic pair, which implies that there is some $v \in V$ which is nonsingular. The following are two easy but useful results which are analogous to Lemma 3.13 and Corollary 3.14 in \cite{HAcase}.

\begin{lemma} \label{lem:Im(phi-bar)}
Let $V = \F{q}^n$, $\phi$ a unitary form on $V$, and $\overline{\phi}$ as in {\rm(\ref{eq:phi-bar})}. Then $\im{\overline{\phi}} = \F{q_0}$, the subfield of index $2$ in $\F{q}$.
\end{lemma}

\begin{proof}
Recall that $f(v,v)^{\sqrt{q}} = f(v,v)$ for any $v \in V$, so $\im{\overline{\phi}} \leq \F{q_0}$. By the preceding remarks $V$ contains a nonsingular vector, say $u$. So $f(\alpha u, \alpha u) = \alpha^{\sqrt{q}+1} f(u,u) = \eta(\alpha) f(u,u)$ for any $\alpha \in \F{q}$, where $\eta : \F{q} \rightarrow \F{q_0}$ is the norm map. Since $\eta$ is surjective so is $\overline{\phi}$, and the result follows.
\end{proof}

If $\phi(v,v) \neq 0$, then $\langle v \rangle^\perp$ is nondegenerate and $V = \langle v \rangle \perp \langle v \rangle^\perp$. On the other hand, if $\phi(v,v) = 0$ then $\langle v \rangle \leq \langle v \rangle^\perp$. By the remarks in \cite[pp. 17--18]{KleidLieb}, the form $\phi$ induces a nondegenerate unitary form $\phi_U$ on the space $U := \langle v \rangle^\perp / \langle v \rangle$, defined by $\phi_U(x + \langle v \rangle, y + \langle v \rangle) := \phi(x,y)$ for all $x,y \in \langle v \rangle^\perp$. It follows from \cite[Propositions 2.1.6 and 2.4.1]{KleidLieb} that all maximal totally isotropic subspaces of $V$ have the same dimension, which, in all cases, is at most $n/2$, so in particular $v^\perp$ contains a nonsingular vector whenever $n \geq 3$.

\begin{corollary} \label{cor:Im(phi-bar)}
Let $V = \F{q}^n$, $\phi$ a unitary form on $V$, $\overline{\phi}$ as in {\rm(\ref{eq:phi-bar})}, and $v \in V^\#$. Then $\im{\overline{\phi}|_{\langle v \rangle^\perp}} = \F{q_0}$ if $v$ is nonsingular and $n \geq 2$, or if $v$ is singular and $n \geq 3$.
\end{corollary}

\begin{proof}
This follows immediately from Lemma \ref{lem:Im(phi-bar)} applied to $\langle v \rangle^\perp$, and the remarks above.
\end{proof}

\begin{proposition} \label{prop:C8-GUdiam2}
Let $\Gamma$ be a graph and $G \leq \Aut{\Gamma}$ such that $G$ satisfies Hypothesis {\rm\ref{hyp:HA}} with $G_0 = \GamU{n,q}$. Then $\Gamma$ is $G$-symmetric with diameter $2$ if and only if $n \geq 2$ and $\Gamma \cong \Cay{V,S}$, where $V = \F{q}^n$ and $S \in \{S_0, S_\#\}$, with $S_0$ and $S_\#$ as in (\ref{eq:S_lambda}) and (\ref{eq:S_theta}), respectively.
\end{proposition}

\begin{proof}
By Lemma \ref{lem:HAgraphs} and Proposition \ref{prop:similarity-orbits} we only need to prove that $\Cay{V,S}$ has diameter $2$ if and only if $n \geq 2$. If $n = 1$ then $V$ is anisotropic, so $\GU{n,q}$ is transitive on $V^\#$ by Proposition \ref{prop:similarity-orbits} (1) and $\Cay{V,S}$ is a complete graph. If $n \geq 2$ then $V^\# \setminus S_0 = S_\#$ and $V^\# \setminus S_\# = S_0$ by Proposition \ref{prop:similarity-orbits}.

\emph{Claim 1: $S_\# \subseteq S_0 + S_0$.} Let $v \in S_\#$. Then by Corollary \ref{cor:Im(phi-bar)} there exists $u \in \langle v \rangle^\perp$ with $\overline{\phi}(u) = -\overline{\phi}(v)$. Set $w := \beta(u+v)$, where $\beta := \alpha\overline{\phi}(v)^{-1}$ and $\alpha \in \F{q}$ such that ${\rm Tr}(\alpha) = \overline{\phi}(v)$. Then $w, v-w \in S_0$, so $v \in S_0 + S_0$ and therefore $S_\# \subseteq S_0 + S_0$.

\emph{Claim 2: $S_0 \subseteq S_\mu + S_\mu$ for any $\mu \in (\im{\overline{\phi}})^\#$.} Let $v \in S_0$. Suppose first that $n \geq 3$. Then by Corollary \ref{cor:Im(phi-bar)}, for any $\mu \in (\im{\overline{\phi}})^\#$ there exists $w \in S_\mu \cap \langle v \rangle^\perp$. It is easy to verify that $\overline{\phi}(v-w) = \overline{\phi}(w)$, so $v-w \in S_\mu$ and $v \in S_\mu + S_\mu$. Therefore $S_0 \subseteq S_\mu + S_\mu$. If $n = 2$ then $\langle v \rangle^\perp = \langle v \rangle$ for any $v \in S_0$. We show that there exists $u \in S_0$ such that $\phi(u,v) = 1$. Indeed, take $x \in V \setminus \langle v \rangle$. Then $\phi(v,x) \neq 0$. If $x \in S_0$ define $u' := x$; if $x \notin S_0$ let $u' := \alpha v + \phi(v,x)x$ where $\alpha \in \F{q}$ with ${\rm Tr}(\alpha) = -\overline{\phi}(x)$. Then in both cases $u' \in S_0$ and $\phi(u',v) \neq 0$, and we take $u$ to be the suitable scalar multiple of $u'$ such that $\phi(u,v) = 1$. Let $w := \beta u + \gamma v$, where $\beta,\gamma \in \F{q}$ with ${\rm Tr}(\beta) = 0$ and ${\rm Tr}(\beta^{q_0}\gamma) = \mu$. Then $w, v-w \in S_\mu$, and thus $v \in S_\mu + S_\mu$. Therefore $S_0 \subseteq S_\mu + S_\mu$.

It follows from Claims 1 and 2, respectively, that $\Cay{V,S_0}$ and $\Cay{V,S_\#}$ both have diameter $2$. This completes the proof.
\end{proof}
\subsection{Class $\C_2$} \label{sec:C2}

In this case $V = \oplus_{i=1}^t{U_i}$, where $U_i = \F{q}^m$ for each $i$, $mt = n$ and $t \geq 2$. Assume that $\mathcal{B} = \bigcup_{i=1}^t{\mathcal{B}_i}$, where $\mathcal{B}_i$ is a basis for $U_i$ for each $i$. We write the elements of $V$ as $t$-tuples over $\F{q}^m$; under this identification the $\tau$-action is equivalent to the natural componentwise action.

Assume first that $H = \GamL{n,q}$. It turns out that the $G_0$-orbits in $V^\#$ are the same as the $L$-orbits, and thus the graphs that we obtain are precisely those in \cite[Proposition 3.2]{HAcase}.

\begin{lemma} \label{lem:C2orbits}
Let $G_0$ be as in case $(\C_2)$ of Theorem {\rm\ref{thm:Asch-GL}}. Then the $G_0$-orbits in $V^\#$ are the sets $X_s$ for each $s \in \{1, \ldots, t\}$, where
	\begin{equation} \label{eq:X_s}
	X_s := \{ (u_1, \ldots, u_t) \in V^\# \ | \ \text{exactly $s$ coordinates nonzero} \}.
	\end{equation}
\end{lemma}

\begin{proof}
Let $v \in X_s$. Clearly $v^{G_0} \subseteq X_s$; since $v^L = X_s$ by \cite[Lemma 3.1]{HAcase} it follows that $v^{G_0} = X_s$.
\end{proof}

\begin{proposition} \label{prop:C2-diam2}
Let $\Gamma$ be a graph and $G \leq \Aut{\Gamma}$ such that $G$ satisfies Hypothesis {\rm\ref{hyp:HA}}, with $H = \GamL{n,q}$ and $G_0$ as in case $(\C_2)$ of Theorem {\rm\ref{thm:Asch-GL}}. Then $\Gamma$ is $G$-symmetric with diameter $2$ if and only if $\Gamma \cong \Cay{V,X_s}$, where $X_s$ is as in {\rm(\ref{eq:X_s})}, such that $q^m > 2$ and $s \geq t/2$.
\end{proposition}

\begin{proof}
This follows immediately from Lemma \ref{lem:C2orbits} and \cite[Proposition 3.2]{HAcase}.
\end{proof}

We now consider the case where $H = \GamSp{n,q}$ with $n \geq 4$. By Theorem \ref{thm:Asch-Sp} there are two types of $\C_2$-subgroups, corresponding to two kinds of decompositions. We refer to these subcases as $(\C_2.1)$ and $(\C_2.2)$.
	\begin{enumerate}[$(\C_2.1)$]
	\item The dimension $m$ of the subspaces $U_i$ is even, $U_i$ is a symplectic space for each $i$, the subspaces $U_i$ are pairwise orthogonal, and
		\begin{align} \label{eq:G0-C2.I}
		G_0 &= \left\{ (g_1, \ldots, g_t)\pi\sigma \ | \ \pi \in \Sym{t},\, \sigma \in \langle \tau \rangle,\, g_i \in \GSp{m,q}, \lambda(g_i) = \lambda(g_1) \right\} \notag \\
		&\cong (\Sp{m,q}^t.[q-1].\Sym{t}) \rtimes \langle \tau \rangle,
		\end{align}
	where $\lambda : \GSp{n,q} \rightarrow \F{q}^\#$ is as defined in Subsection \ref{subsec:semilin}.
	\item The dimension $m = n/2$ so that $t=2$, both subspaces $U_i$ are totally singular with dimension $n/2$, $q$ is odd if $n = 4$, and
		\begin{align} \label{eq:G0-C2.II}
		G_0 &= \left\{ \left(g,g^{-\top}\right)\pi\sigma \ \vline \ \pi \in \Sym{t},\, \sigma \in \langle \tau \rangle,\, g \in \GL{m,q} \right\} \notag \\
		&\cong (\GL{m,q}.[2]) \rtimes \langle \tau \rangle,
		\end{align}
	where $g^\top$ denotes the transpose of $g$, and $g^{-\top} = (g^\top)^{-1}$.
	\end{enumerate}

\begin{lemma} \label{lem:C2orbits-Sp}
For each $s \in \{1, \ldots, t\}$ let $X_s$ be as in {\rm(\ref{eq:X_s})}. The $G_0$-orbits in $V^\#$ are
	\begin{enumerate}
	\item the sets $X_s$ for each $s \in \{1, \ldots, t\}$ if case $(\C_2.1)$  holds and $G_0$ is as in {\rm(\ref{eq:G0-C2.I})};
	\item the sets $X_1$ and $\bigcup_{\sigma \in \langle \tau \rangle}{W_{\beta^\sigma}}$ for all $\beta \in \F{q}$, if case $(\C_2.2)$  holds and $G_0$ is as in {\rm(\ref{eq:G0-C2.II})}, where
		\begin{equation} \label{eq:W_beta}
		W_\beta := (w_1,x_\beta)^L,
		\end{equation}
$L = G_0 \cap \GL{n,q} \cong \GL{m,q}.[2]$, $w_1 := (1,0, \ldots, 0) \in \F{q}^m$, and $x_\beta \in (\F{q}^m)^\#$ with first component $\beta$.
	\end{enumerate}
\end{lemma}

\begin{proof}
The proof of part (1) is similar to that of \cite[Lemma 3.1]{HAcase} and uses the transitivity of $\Sp{m,q}$ on $U_i^\#$, so we only need to prove part (2). Assume that case ($\C_2.2$) holds. Then $L = K.\Sym{2}$, where $K := \left\{ \left(g,g^{-\top}\right) \ \vline \ g \in \GL{m,q} \right\}$. It is easy to see that $U_1 \oplus \{\0\}$ and $\{\0\} \oplus U_2$ are $K$-orbits, so $X_1 = (U_1 \otimes \{\0\}) \cup (\{\0\} \oplus U_2)$ is a $G_0$-orbit. Let $(u,v) \in X_2$, and for any $\beta \in \F{q}$ define
	\begin{equation} \label{eq:w_beta}
	w_\beta :=
		\begin{cases}
		(\beta, 0, \ldots, 0) \ &\text{if $\beta \neq 0$}, \\
		(0, 1, 0, \ldots, 0) &\text{if $\beta = 0$}.
		\end{cases}
	\end{equation}
Since $w_1 \in u^{\GL{m,q}}$ we can assume that $u = w_1$. Suppose that $v = (\beta,v_2,\ldots,v_m)$.

\emph{Claim 1: $(w_1,y) \in (w_1,v)^K$ if and only if $y = (\beta,y_2,\ldots,y_m)$ for some $y_2, \ldots, y_m \in \F{q}$.} Indeed, $(w_1,y) \in (w_1,v)^K$ if and only if $y = v^{h^{-\top}}$ for some $h \in \Stab{\GL{m,q}}{w_1}$. Now $w_1^h = w_1$ if and only if the matrix of $h^{-\top}$ has the form
	\[ \renewcommand{\arraystretch}{0.8}
	\left( \begin{array}{c|ccc}
	1 &  & C &  \\
	\hline 0 &  &  &  \\
	\vdots &  & D &  \\
	0 &  &  &
	\end{array} \right) \]
where $C$ is a $1 \times (m-1)$ matrix over $\F{q}$ and $D \in \GL{m-1,q}$. Clearly, the orbit of $v$ under the subgroup $\left\{ h^{-\top} \ \vline \ h \in \Stab{\GL{m,q}}{w_1} \right\}$ is the set of all nonzero vectors in $\F{q}^m$ with first component $\beta$. Therefore Claim 1 holds.

\emph{Claim 2: $(w_1,v)^L = (w_1,v)^K$.} By Claim 1 we can assume that $v = w_\beta$. If $\beta \neq 0$ let
	\[ \renewcommand{\arraystretch}{0.8}
	g := \left( \begin{array}{c|ccc}
	\beta & 0 & \cdots & 0 \\
	\hline 0 &  &  &  \\
	\vdots & \multicolumn{3}{c}{I_{m-1}} \\
	0 &  &  &
	\end{array} \right). \]
If $\beta = 0$ let $g :=
	\left( \renewcommand{\arraystretch}{0.8} 
	\begin{matrix}
	0 & 1 \\
	1 & 0
	\end{matrix} \right)$
if $m = 2$, and
	\[ \renewcommand{\arraystretch}{0.8}
	g := \left( \begin{array}{cc|ccc}
	0 & 1 &  &  &  \\
	1 & 0 &  & 0 &  \\
	\hline  &  &  &  &  \\
	\multicolumn{2}{c|}{0} &  & I_{m-2} &  \\
	 &  &  &  &  \\
	\end{array} \right) \]
if $m > 2$. Then $g \in \GL{m,q}$ for all cases, and $w_1^g = w_1^{g^\top} = v$. Hence $(w_1^g,v^{g^{-\top}}) = (v,w_1)$, so that $(v,w_1) \in (w_1,v)^K$. Therefore $(w_1,v)^L = (w_1,v)^K \cup (v,w_1)^K = (w_1,v)^K$, which proves Claim 2.

It follows from Claims 1 and 2 that each set $W_\beta$ is an $L$-orbit (and moreover $W_\beta = W_{\beta'}$ if and only if $\beta = \beta'$). It follows from the definition of the $\tau$-action on $V^\#$ that $(w_1,v)^{G_0} = \bigcup_{\sigma \in \langle \tau \rangle}{W_{\beta^\sigma}}$. This completes the proof of part (2).
\end{proof}

\begin{proposition} \label{prop:C2-diam2-Sp}
Let $\Gamma$ be a graph and $G \leq \Aut{\Gamma}$ such that $G$ satisfies Hypothesis {\rm\ref{hyp:HA}} with $H = \GamSp{n,q}$ and $i = 2$. Then $\Gamma$ is $G$-symmetric with diameter $2$ if and only if $\Gamma \cong \Cay{V,S}$, where
	\begin{enumerate}
	\item if case $(\C_2.1)$  holds, then $q^m > 2$, $G_0$ is as in {\rm(\ref{eq:G0-C2.I})}, $S = X_s$, and $s \geq t/2$;
	\item if case $(\C_2.2)$  holds with $q^m = 2$, then $G_0$ is as in {\rm(\ref{eq:G0-C2.II})}, and $S = W_\beta$ for any $\beta \in \F{q}$;
	\item if case $(\C_2.2)$  holds with $q^m > 2$, then $G_0$ is as in {\rm(\ref{eq:G0-C2.II})}, and $S = X_1$ or $S = \bigcup_{\sigma \in \langle \tau \rangle}{W_{\beta^\sigma}}$ for some $\beta \in \F{q}$;
	\end{enumerate}
with $X_s$ as in {\rm(\ref{eq:X_s})} and $W_\beta$ as {\rm(\ref{eq:W_beta})}.
\end{proposition}

\begin{proof}
The graph of (1) is precisely that of Proposition {\rm\ref{prop:C2-diam2}}, and the fact that it is $G$-symmetric follows from Lemma \ref{lem:HAgraphs}. So assume that case ($\C_2.2$) holds. By Lemma \ref{lem:HAgraphs} we only need to show that $V = S \cup (S+S)$ unless $S = X_1$ and $q = 2$. It follows from Proposition {\rm\ref{prop:C2-diam2}} (with $t = 2$) that $\Cay{V,X_1}$ has diameter $2$ (with $G$ quasiprimitive) if and only if $q^m > 2$, which proves part of statement (3). Thus we may assume that $S = \bigcup_{\sigma \in \langle \tau \rangle}{W_{\beta^\sigma}}$ for some $\beta \in \F{q}$. It remains to prove that $V = S \cup (S+S)$.

Let $w_\beta$ be as in (\ref{eq:w_beta}) and $\gamma \in \F{q}$, with $\gamma \neq \beta$. Define
	\[ g_0 :=
	\left( 
	\begin{array}{cc}
	1 & 1 \\
	0 & 1
	\end{array} \right)
\text{ and } \
	h_0 :=
	\left( 
	\begin{matrix}
	0 & -1 \\
	-1 & \gamma_0
	\end{matrix} \right), \]
where $\gamma_0 := 1 - \beta^{-1}\gamma$ if $\beta \neq 0$ and $\gamma_0 := 0$ if $\beta = 0$. If $m = 2$ let $g := g_0$ and $h := h_0$; if $m \geq 3$ define $g$ and $h$ by
	\[ g := \left( 
	\begin{matrix}
	g_0 & 0 \\
	0 & I_{m-2}
	\end{matrix}
	\right) \]
and
	\[ h := \left( \renewcommand{\arraystretch}{0.8}
	\begin{array}{c|ccc}
	& 0 & \cdots & 0 \\
	h_0 & 1 & \cdots & 1 \\
	\hline \\
	0 & \multicolumn{3}{c}{I_{m-2}} \\
	&  &  &
	\end{array}
	\right). \]
Then $g,h \in \GL{m,q}$ for all $m \geq 2$, and $w_1^g + w_1^h = w_1$. Recall that $q$ is odd if $m = 2$, so we can take $x \in (\F{q}^m)^\#$ where
	\[ 
	x := \left\{ \begin{array}{ll}
	w_\beta \ &\text{if $\beta \neq 0$}; \\[2pt]
	(0, -\gamma/2) &\text{if $\beta = 0$ and $m = 2$}; \\[2pt]
	(0,0,1,0,\ldots,0) &\text{if $\beta = 0$ and $m \geq 3$}.
	\end{array} \right. \]
Then for all cases $y := x^{g^{-\top}} + x^{h^{-\top}}$ has first component $\gamma$. Hence, applying Lemma \ref{lem:C2orbits-Sp}, we have $W_\gamma = (w_1,y)^L \subseteq W_\beta + W_\beta$ for any $\gamma \neq \beta$. Since also $\{\0\} \cup X_1 \subseteq W_\beta + W_\beta$, it follows that $V = W_\beta \cup (W_\beta + W_\beta)$. Therefore $V = S \cup (S+S)$, which completes the proof of parts (2) and (3).
\end{proof}

\subsection{Class $\C_4$} \label{sec:C4}

In this case $V = U \otimes W = \F{q}^k \otimes \F{q}^m$ with $k,m \geq 2$, and $\mathcal{B}$ is a tensor product basis of $V$, that is,
	\begin{equation*}
	\mathcal{B} = \{ u_i \otimes w_j \ | \ 1 \leq i \leq k, \ 1 \leq j \leq m \},
	\end{equation*}
where $\mathcal{B}_U := \{u_1, \ldots, u_k\}$ and $\mathcal{B}_W := \{w_1, \ldots, w_m\}$ are fixed bases of $U$ and $W$, respectively. We choose $\tau$ to fix each of the vectors $u_i \otimes w_j$. Then for any simple vector $u \otimes w \in V$, we have $(u \otimes w)^\tau = u^\tau \otimes w^\tau$, and for any $v = \sum_{i=1}^{r}{(a_i \otimes b_i)} \in V$,
	\[ v^\tau = \sum_{i=1}^r{a_i^\tau \otimes b_i^\tau}. \]
Recall that $k \neq m$ in the description given in Theorems \ref{thm:Asch-GL} and \ref{thm:Asch-Sp}; however, all of the results in this section also hold for $k = m$, so we do not assume that $k$ and $m$ are distinct. In this way the results yield useful information for the $\C_7$ case.

A nonzero vector in $V$ is said to be \emph{simple} in the decomposition $U \otimes W$ if it can be written as $u \otimes w$ for some $u \in U$ and $w \in W$. The \emph{tensor weight} $\wt{v}$ of $v \in V^\#$, with respect to this decomposition, is the least number $s$ such that $v$ can be written as the sum of $s$ simple vectors in $U \otimes W$. It follows from \cite[Lemma 3.3]{HAcase} that $\wt{v} \leq \min{k,m}$ for any $v \in V^\#$, and that for each $s \in \{ 1, \ldots, \min{k,m} \}$ there is a vector $v \in V^\#$ with weight $s$.

The proof of the following observation is straightforward and is omitted.

\begin{lemma} \label{lem:tensorwt-aut}
For any $v \in V^\#$ and any $\sigma \in \Aut{\F{q}}$,
	\[ \wt{v^\sigma} = \wt{v}. \]
\end{lemma}

Assume first that $H = \GamL{n,q}$. As in the previous section, the $G_0$-orbits in $V^\#$ are the same as the $L$-orbits. This follows easily from Lemma \ref{lem:tensorwt-aut} and the results in \cite{HAcase}.

\begin{lemma} \label{lem:C4orbits}
Let $G_0$ be as in case {\rm($\C_4$)} of Theorem {\rm\ref{thm:Asch-GL}}. Then the $G_0$-orbits in $V^\#$ are the sets $Y_s$ for each $s \in \{1, \ldots, \min{k,m}\}$, where
	\begin{equation} \label{eq:Y_s}
	Y_s := \left\{ v \in V^\# \ \big| \ \wt{v} = s \right\}.
	\end{equation}
\end{lemma}

\begin{proof}
This is a consequence of Lemma \ref{lem:tensorwt-aut} above, and of \cite[Lemmas 3.3 and 3.4]{HAcase}.
\end{proof}

We then obtain the same graphs as those in \cite[Proposition 3.5]{HAcase}.

\begin{proposition} \label{prop:C4-diam2}
Let $\Gamma$ be a graph and $G \leq \Aut{\Gamma}$ such that $G$ satisfies Hypothesis {\rm\ref{hyp:HA}} with $G_0$ as in case {\rm($\C_4$)} of Theorem {\rm\ref{thm:Asch-GL}}, where $k$ and $m$ may be equal. Then $\Gamma$ is $G$-symmetric with diameter $2$ if and only if $\Gamma \cong \Cay{V,Y_s}$, where $s \geq \frac{1}{2}\min{k,m}$ and $Y_s$ is as in {\rm(\ref{eq:Y_s})}.
\end{proposition}

\begin{proof}
This follows immediately from Lemma \ref{lem:C4orbits} and \cite[Propositon 3.5]{HAcase}.
\end{proof}

Now assume that $H = \GamSp{n,q}$. In this case $k$ is even, $m \geq 3$, $q$ is odd, and $\phi = \phi_U \otimes \phi_W$, where $\phi_U$ is a symplectic form on $U$ and $\phi_W$ is a nondegenerate symmetric bilinear form on $W$. We can choose $B_U$ and $B_W$ appropriately so that $B$ is a symplectic basis and hence we can again choose $\tau$ to fix each of the vectors $u_i \otimes w_j$. The $G_0$-orbits in this case are proper subsets of the sets $Y_s$ in (\ref{eq:Y_s}), and are in general rather difficult to describe, as are the $L$-orbits. For instance, if $v = \sum_{i=1}^s{a_i \otimes b_i} \in Y_s$, it is easy to see that
	\[ v^{G_0} = \left\{ \sum_{i=1}^s{a'_i \otimes b'_i} \ \Big| \ a'_i \in U^\#, \ b'_i \in b_i^{{\rm GO}^\epsilon(m,q)} \right\}. \]
If $s = 1$ then the set $Y_1$ of simple vectors splits into the $G_0$-orbits $Y_1^\theta$, where $\theta \in \{0,\#\}$ if $m$ is even and $\theta \in \{0,\square,\boxtimes\}$ if $m$ is odd, and
	\[ Y_1^\theta := \left\{ a \otimes b \ \big| \ a \in U^\#, \ b \in S_\theta \right\}. \]
If $s > 1$ suppose that exactly $r$ of the vectors $b_i$ belong in $S_\#$ for some $r$, $0 \leq r \leq s$; if $m$ is odd suppose further that exactly $r_\square$ belong in $S_\square$ and $r_\boxtimes$ in $S_\boxtimes$. If $m$ is even then $v^{G_0} \subset Y_s^r$, where
	\[ Y_s^r := \left\{ \sum_{i=1}^s{a'_i \otimes b'_i} \in Y_s \ \Big| \ \text{exactly $r$ of the vectors $b'_i$ are in $S_\#$} \right\}, \]
and if $m$ is even then $v^{G_0} \subset Y_s^{r_\square,r_\boxtimes}$, where
	\[ Y_s^{r_\square,r_\boxtimes} := \left\{ \sum_{i=1}^s{a'_i \otimes b'_i} \in Y_s \ \Big| \ \text{exactly $r_\theta$ of the vectors $b'_i$ are in $S_\theta$ for $\theta \in \{\square,\boxtimes\}$} \right\}. \]
The sets $Y_s^r$ and $Y_s^{r_\square,r_\boxtimes}$ above are, in general, not $G_0$-orbits. For instance, if $s = 2$, the weight-2 vectors $a_1 \otimes b_1 + a_2 \otimes b_2, a'_1 \otimes b'_1 + a'_2 \otimes b'_2 \in Y_2^0$ (or $Y_2^{0,0}$ if $m$ is even), such that $b_1 \perp b_2$ and $b'_1 \not\perp b'_2$, belong to different $G_0$-orbits.

The following is an easy consequence of the preceding discussion. However, as discussed, we do not have a good description of the $G_0$-orbits.

\begin{proposition} \label{prop:C4-diam2-Sp}
Let $\Gamma$ be a graph and $G \leq \Aut{\Gamma}$ such that $G$ satisfies Hypothesis {\rm\ref{hyp:HA}} with $G_0$ as in case {\rm($\C_4$)} of Theorem {\rm\ref{thm:Asch-Sp}}, where $k$ and $m$ may be equal. If $\Gamma$ is $G$-symmetric with diameter $2$, then $\Gamma \cong \Cay{V,S}$ where $S = v^{G_0}$ for some $v \in Y_s$, where $Y_s$ is as in {\rm(\ref{eq:Y_s})} and $s \geq \frac{1}{2}\min{k,m}$.
\end{proposition}

\begin{proof}
This follows immediately from the discussion above together with Proposition \ref{prop:C4-diam2}.
\end{proof}
\subsection{Class $\C_5$} \label{sec:C5}

In this case $n \geq 2$, $d/n$ is composite with a prime divisor $r$, and $V$ has a fixed ordered basis
	\begin{equation*} 
	\mathcal{B} := (v_1, \ldots, v_n).
	\end{equation*}
Let $q_0 := q^{1/r}$ and let $\F{q_0}$ denote the subfield of $\F{q}$ of index $r$. Let $V_0$ be the $\F{q_0}$-span of $\mathcal{B}$. Then $V_0$ is a vector space over $\F{q_0}$ that is contained in $V$, but $V_0$ is not an $\F{q}$-subspace of $V$.

To any $v = \sum_{i=1}^n{\alpha_i v_i} \in V$ we can associate the $\F{q_0}$-subspace $D_v$ of $\F{q}$, where
	\begin{equation} \label{eq:D_v}
	D_v := \langle \alpha_1, \ldots, \alpha_n \rangle_{\F{q_0}}.
	\end{equation}
Set
	\begin{equation} \label{eq:c(v)}
	c(v) := \Dim{D_v}{\F{q_0}},
	\end{equation}
and note that $c(v) \leq \min{r,n}$. For any $\lambda \in \F{q}$ it is clear that $D_{\lambda v} = \lambda D_v$, so $c(\lambda v) = c(v)$, and it is also easy to show that $c(v^\sigma) = c(v)$ for any $\sigma \in \Aut{\F{q}}$. Let
	\[ [D_v] := \{ \lambda D_v \ | \ \lambda \in \F{q}^\# \}, \]
and observe that $D_u \in [D_{v^\sigma}]$ if and only if $D_u = \lambda D_{v^\sigma} = \left(\lambda^{\sigma^{-1}} D_v\right)^\sigma$ for some $\lambda \in \F{q}^\#$. Hence $D_{u^{\sigma^{-1}}} = (D_u)^{\sigma^{-1}} = \lambda^{\sigma^{-1}} D_v$, so that $D_{u^{\sigma^{-1}}} \in [D_v]$. Thus $[D_{v^\sigma}] = [D_v]^\sigma$.


\subsubsection{Case $H = \GamL{n,q}$}

By Theorem \ref{thm:Asch-GL}
	\[ G_0 = (\GL{n,q_0} \circ Z_{q-1}) \rtimes \langle \tau \rangle \]
and $L = \GL{n,q_0} \circ Z_{q-1}$.

Regard the field $\F{q}$ as a vector space of dimension $r$ over $\F{q_0}$, and for any $a \in \{1, \ldots, r\}$, define
	\begin{equation} \label{eq:fieldofD}
	\K{a} := \begin{cases}
	           \F{q} \ &\text{if $a = r$}, \\
		   \F{q_0} &\text{otherwise.}
	           \end{cases}
	\end{equation}
For $a \in \{1, \ldots, r\}$ define
	\begin{equation} \label{eq:eta} \renewcommand{\arraystretch}{0.8}
	\eta(a) :=
	\frac{\left[ \begin{array}{c}
	r \\ a
	\end{array} \right]_{q_0}}{\big|\F{q}^\# : \K{a}^\#\big|},
	\end{equation}
where
	\[ \renewcommand{\arraystretch}{0.8}
	\left[ \begin{array}{c}
	r \\ a
	\end{array} \right]_{q_0} := \prod_{i=0}^{a-1}{\frac{q_0^r - q_0^i}{q_0^a - q_0^i}}, \]
the number of $a$-dimensional subspaces of $\F{q_0}^r$. In particular $\eta(r) = \eta(1) = 1$. Lemma \ref{lem:K-eta} gives some elementary observations about $\K{a}$ and $\eta$, whose significance will be apparent in Corollary \ref{cor:numberoforbits}. The proof of Lemma \ref{lem:K-eta} is straightforward and is omitted.

\begin{lemma} \label{lem:K-eta}
Let $\F{q_0}$ be a proper nontrivial subfield of $\F{q}$ with prime index $r$, and suppose that $\F{q}$ is viewed as a vector space over $\F{q_0}$ with dimension $r$. For any $a \in \{1, \ldots, r\}$, let $\mathcal{D}$ denote the set of all $\F{q_0}$-subspaces of $\F{q}$ with dimension $a$, and let $\K{a}$ and $\eta(a)$ be as defined in {\rm(\ref{eq:fieldofD})} and {\rm(\ref{eq:eta})}, respectively. Then the following hold:
	\begin{enumerate}
	\item For any $D \in \mathcal{D}$,
		\[ \{ \lambda \in \F{q} \ | \ \lambda D = D \} = \K{a}. \]
	\item For any $D \in \mathcal{D}$, the sets $[D] = \big\{ \lambda D \ | \ \lambda \in \F{q}^\# \big\}$ partition $\mathcal{D}$. Moreover, $|[D]| = \big|\F{q}^\# : \K{a}^\#\big|$, and the number of distinct parts $[D]$ in $\mathcal{D}$ is $\eta(a)$.
	\end{enumerate}
\end{lemma}

The main result for this case, which relies on the value of the parameter $c(v)$, is the following. It shows that examples do exist.

\begin{proposition} \label{prop:C5main}
Let $\Gamma$ be a graph and $G \leq \Aut{\Gamma}$ such that $G$ satisfies Hypothesis {\rm\ref{hyp:HA}} with $H = \GamL{n,q}$ and $i = 5$. Then $\Gamma$ is connected and $G$-symmetric if and only if $\Gamma \cong \Cay{V,v^{G_0}}$ for some $v \in V^\#$. Moreover, if $D_v$ and $c(v)$ are as in {\rm(\ref{eq:D_v})} and {\rm(\ref{eq:c(v)})}, respectively, then the following hold.
	\begin{enumerate}
	\item If $c(v) = r$ or $c(v) = r-1$ then $\diam{\Gamma} = 2$.
	\item If $c(v) = 1$ then $\diam{\Gamma} = \min{n,r}$. In particular $\diam{\Gamma} = 2$ if and only if $n = 2$ or $r = 2$.
	\item If $2 \leq c(v) < \frac{1}{2} \min{n,r}$ then $\diam{\Gamma} > 2$.
	\item Let $\eta$ be as defined in {\rm(\ref{eq:eta})}, $s$ be the largest divisor of $d/n$ with $s \leq \eta(c(v))$, and
		\begin{equation*} \label{eq:k1(q_0)}
		k_1(q_0) := \begin{cases} 18s/17 \ &\text{if } q_0 = 2; \\ s - 5/4 &\text{if } q_0 > 2. \end{cases} 
		\end{equation*}
	If $3 \leq n < r$ and $n/2 \leq c(v) < (r(n-2) + k_1(q_0))/(2n)$, then $\diam{\Gamma} > 2$.
	\end{enumerate}
\end{proposition}

The cases not covered by Proposition \ref{prop:C5main} are discussed briefly at the end of the section. The proof of Proposition \ref{prop:C5main} is given after Lemma \ref{lem:C5diam2}, and relies on several intermediate results. We begin by describing the $\GL{n,q_0}$-orbits in terms of the subspaces $D_v$, which in turn leads to a description of the $G_0$-orbits in $V^\#$.

\begin{lemma} \label{lem:GLorbits}
For any $v \in V^\#$ let $D_v$ and $c(v)$ be as in {\rm(\ref{eq:D_v})} and {\rm(\ref{eq:c(v)})}, respectively, and let $\mathcal{U}$ denote the set of all $\F{q_0}$-independent $c(v)$-tuples in $V_0$. Then for any fixed $\F{q_0}$-basis $\{\beta_1, \ldots, \beta_{c(v)}\}$ of $D_v$,
	\begin{align*}
	v^{\GL{n,q_0}} &= \left\{ \sum_{i=1}^{c(v)}{\beta_i u_i} \ \Big| \ (u_1, \ldots, u_{c(v)}) \in \mathcal{U} \right\} \\
	&= \left\{ u \in V^\# \ \big| \ D_u = D_v \right\}.
	\end{align*}
\end{lemma}

\begin{proof}
Suppose that $v = \sum_{i=1}^n{\alpha_i v_i}$. Define
	\begin{equation} \label{eq:GLorbits-X}
	U := \left\{ u \in V^\# \ \big| \ D_u = D_v \right\}
	\end{equation}
and
	\begin{equation} \label{eq:GLorbits-Y}
	W := \left\{ \sum_{i=1}^{c(v)}{\beta_i u_i} \ \Big| \ (u_1, \ldots, u_{c(v)}) \in \mathcal{U} \right\}.
	\end{equation}

\emph{Claim 1: $v^{\GL{n,q_0}} \subseteq U$.} Let $g \in \GL{n,q_0}$ with matrix $\left[g_{jk}\right]$ with respect to $\mathcal{B}$. Then $v^g = \sum_{k=1}^n{\alpha'_k v_k}$, where $\alpha'_k = \sum_{j=1}^n{\alpha_j g_{jk}} \in D_v$ for each $k$. Hence $D_{v^g} \leq D_v$. Since $v$ and $g$ are arbitrary, we also have $D_v \leq D_{v^g}$. So $D_{v^g} = D_v$, and therefore $v^{\GL{n,q_0}} \subseteq U$.

\emph{Claim 2: $U \subseteq W$.} Let $u = \sum_{j=1}^n{\alpha'_j v_j} \in U$. Writing $\alpha'_j = \sum_{i=1}^{c(v)}{\beta_i \gamma_{ij}}$ for each $j$, where all $\gamma_{ij} \in \F{q_0}$, we get $u = \sum_{i=1}^{c(v)}{\beta_i u_i}$, with $u_i = \sum_{j=1}^n{\gamma_{ij} v_j} \in V_0$ for all $i$. It remains to show that the set $\mathbf{u} := \{u_1, \ldots, u_{c(v)}\}$ is $\F{q_0}$-independent. Indeed, let $\{u'_1, \ldots, u'_b\}$ be a maximal $\F{q_0}$-independent subset of $\mathbf{u}$, and extend this to an ordered $\F{q_0}$-basis $\mathcal{B}' := (u'_1, \ldots, u'_d)$ of $V_0$. Then $u = \sum_{k=1}^b{\beta'_k u'_k}$ for some $\beta'_1, \ldots, \beta'_b \in \F{q}$, and if $g \in \GL{n,q_0}$ is the change of basis matrix from $\mathcal{B}'$ to $\mathcal{B}$, then $u^g = \sum_{k=1}^b{\beta'_k v_k}$. So $D_u = D_{u^g}$ by Claim 1, and thus $b \leq c(v) = \Dim{D_u}{\F{q_0}} = \Dim{D_{u^g}}{\F{q_0}} \leq b$. Hence $b = c(v)$ and $\mathbf{u}$ is $\F{q_0}$-independent. Therefore $U \subseteq W$.

\emph{Claim 3: $W \subseteq v^{\GL{n,q_0}}$.} It is easy to see that $W$ is contained in one orbit of $\GL{n,q_0}$, and it follows from Claims 1 and 2 that $v \in W$. So $W \subseteq v^{\GL{n,q_0}}$, as claimed.

Thus we have $v^{\GL{n,q_0}} = U = W$ by Claims 1 -- 3.
\end{proof}

\begin{proposition} \label{prop:C5orbits}
For any $v \in V^\#$ let $D_v$ and $c(v)$ be as in {\rm(\ref{eq:D_v})} and {\rm(\ref{eq:c(v)})}, respectively, and let $\mathcal{U}$ be the set of all $\F{q_0}$-independent $c(v)$-tuples in $V_0$. Then for any fixed $\F{q_0}$-basis $\{\beta_1, \ldots, \beta_{c(v)}\}$ of $D_v$ we have
	\begin{align*}
	v^L &= \left\{ \lambda \sum_{i=1}^{c(v)}{\beta_i u_i} \ \Big| \ (u_1, \ldots, u_{c(v)}) \in \mathcal{U}, \, \lambda \in \F{q}^\# \right\} \\
	&= \left\{ u \in V^\# \ \big| \ D_u = \lambda D_v, \, \lambda \in \F{q}^\# \right\}
	\end{align*}
and
	\begin{align} \label{eq:C5orbits}
	v^{G_0} &= \left\{ \lambda \sum_{i=1}^{c(v)}{\beta_i^\sigma u_i} \ \Big| \ \ (u_1, \ldots, u_{c(v)}) \in \mathcal{U},\, \lambda \in \F{q}^\#, \, \sigma \in \langle \tau \rangle \right\} \notag \\
	&= \left\{ u \in V^\# \ \big| \ D_u = \lambda (D_v)^\sigma, \, \lambda \in \F{q}^\#, \sigma \in \langle \tau \rangle \right\}.
	\end{align}
\end{proposition}

\begin{proof}
Let $U' := \{ u \in V^\# \ | \ D_u = \lambda D_v \ \text{for some} \ \lambda \in \F{q}^\# \}$. Since $L = \GL{n,q_0} \circ Z_{q-1}$ and $D_{\lambda v} = \lambda D_v$ for any $\lambda \in \F{q}$, it follows from Lemma \ref{lem:GLorbits} that $v^L = U'$.
Thus
	\[ v^{G_0} = \bigcup_{\sigma \in \langle \tau \rangle}{\left\{ u^\sigma \ \big| \ u \in v^L \right\}} \subseteq W', \]
where $W' := \big\{ u \in V^\# \ \big| \ D_u = \lambda (D_v)^\sigma, \, \lambda \in \F{q}^\#, \, \sigma \in \langle \tau \rangle \big\}$. For any $w \in W$ with $D_w = \mu (D_v)^\rho$ for $\mu \in \F{q}^\#$ and $\rho \in \langle \tau \rangle$, we have $w \in (v^\rho)^L \subseteq v^{G_0}$. Therefore $v^{G_0} = W'$, and the rest follows from Lemma \ref{lem:GLorbits}.
\end{proof}

\begin{corollary} \label{cor:numberoforbits}
Let $v \in V^\#$, and let $\mathbb{K}$, $\eta$, $D_v$ and $c(v)$ be as defined in {\rm(\ref{eq:fieldofD})}, {\rm(\ref{eq:eta})}, {\rm(\ref{eq:D_v})} and {\rm(\ref{eq:c(v)})}, respectively.
	\begin{enumerate}
	\item For $a \in \{1, \ldots, \min{n,r}\}$, the number of orbits $v^L$ with $c(v) = a$ is $\eta(a)$.
	\item $\left|v^L\right| = \renewcommand{\arraystretch}{0.8} \left[ \begin{array}{cc} n \\ c(v) \end{array} \right]_{q_0} \cdot \big|\GL{c(v),q_0}\big| \cdot \big|\F{q}^\# : \K{c(v)}^\#\big|$
	\item $\left|v^{G_0}\right| = s\left|v^L\right|$ for some divisor $s$ of $d/n$ with $s \leq \eta(c(v))$.
	\end{enumerate}
\end{corollary}

\begin{proof}
It follows from Proposition \ref{prop:C5orbits} that the map $v^L \mapsto [D_v] := \{ \lambda D_v \ | \ \lambda \in \F{q}^\# \}$ is a one-to-one correspondence between the set of $L$-orbits and the set of classes $[D]$ of $\F{q_0}$-subspaces of $\F{q}$. Therefore, by Lemma \ref{lem:K-eta} (2), there are exactly $\eta(a)$ orbits $v^L$ with $c(v) = a$, which proves part (1). Also by Proposition \ref{prop:C5orbits}, we have $|v^L| = |\mathcal{U}| \cdot |[D_v]|$, where $\mathcal{U}$ is the set of $\F{q_0}$-independent $c(v)$-tuples in $V_0$. So
	\[ |\mathcal{U}| = \renewcommand{\arraystretch}{0.8} \left[ \begin{array}{cc} n \\ c(v) \end{array} \right]_{q_0} |\GL{c(v),q_0}|, \]
and by Lemma \ref{lem:K-eta} (2), $|[D_v]| = |\F{q}^\# : \K{c(v)}^\#|$. This proves part (2). Since $L \lhd G_0$ we must have $\left|v^{G_0}\right| = s\left|v^L\right|$ for some $s$ dividing $|G_0 : L| = |\Aut{\F{q}}| = d/n$. Also $s \leq \eta(c(v))$ since $c(v^\sigma) = c(v)$, which proves part (3).
\end{proof}

\begin{lemma} \label{lem:C5diam2}
Let $\Gamma = \Cay{V,v^{G_0}}$ for some $v \in V^\#$, and let $c(v)$ be as in {\rm(\ref{eq:c(v)})}. Let $w \in V$.
	\begin{enumerate}
	\item If $w \in v^{G_0} + v^{G_0}$ then $c(w) \leq 2c(v)$.
	\item If $D_w < D_v$ then $w \in v^{G_0} + v^{G_0}$.
	\end{enumerate}
\end{lemma}

\begin{proof}
Let $\mathcal{U}$ and $\mathcal{W}$ denote the sets of $\F{q_0}$-independent $c(v)$- and $c(w)$-tuples, respectively, in $V$.

Suppose first that $w = x+y$ for some $x,y \in v^{G_0}$. Then by Proposition \ref{prop:C5orbits} we can write $x$ and $y$ as $x = \sum_{i=1}^{c(v)}{\lambda \beta_i^\rho x_i}$ and $y = \sum_{i=1}^{c(v)}{\mu \beta_i^\sigma y_i}$ for some scalars $\lambda,\mu \in \F{q}^\#$, maps $\rho,\sigma \in \Aut{\F{q}}$, and $c(v)$-tuples $\left(x_1, \ldots, x_{c(v)}\right), \left(y_1, \ldots, y_{c(v)}\right) \in \mathcal{U}$. Hence
	\[ D_w = D_{x+y} \subseteq \left\langle \lambda \beta_1^\rho, \ldots, \lambda \beta_{c(v)}^\rho, \mu \beta_1^\sigma, \ldots, \mu \beta_{c(v)}^\sigma \right\rangle_{\F{q_0}}, \]
and therefore $c(w) = c(x+y) \leq 2c(v)$. This proves part (1).

To prove part (2), observe that Lemma \ref{lem:GLorbits} implies that we can write $v$ and $w$ as $v = \sum_{i=1}^{c(v)}{\gamma_i u_i}$ and $w = \sum_{i=1}^{c(w)}{\delta_i z_i}$ for some $\left(u_1, \ldots, u_{c(v)}\right) \in \mathcal{U}$ and $\left(z_1, \ldots, z_{c(w)}\right) \in \mathcal{W}$, and for some fixed $\F{q_0}$-bases $\{\gamma_i, \ldots, \gamma_{c(v)}\}$ and $\{\delta_1, \ldots, \delta_{c(w)}\}$ of $D_v$ and $D_w$, respectively. Since $D_w < D_v$ then $c(w) < c(v)$, and we can extend $\{\delta_1, \ldots, \delta_{c(w)}\}$ to an $\F{q_0}$-basis $\{\delta_1, \ldots, \delta_{c(v)}\}$ of $D_v$, and $\left(z_1, \ldots, z_{c(w)}\right)$ to $\left(z_1, \ldots, z_{c(v)}\right) \in \mathcal{U}$. Set $x := \sum_{i=1}^{c(v)}{\delta_i z_i}$ and $y := \sum_{i=1}^{c(v)}{\delta_i y_i}$, where $y_i := z_{i+1} - z_i$ if $1 \leq i \leq c(w)-1$, $y_{c(w)} := z_1$, and $y_i := -z_i$ if $c(w)+1 \leq i \leq c(v)$. Then $\left(y_1, \ldots, y_{c(v)}\right) \in \mathcal{U}$ and $D_x = D_y = D_v$, so by Lemma \ref{lem:GLorbits} we have $x,y \in v^{\GL{n,q_0}} \subseteq v^{G_0}$. Therefore $x+y \in v^{G_0} + v^{G_0}$. Now $D_w = D_{x+y}$, so applying Lemma \ref{lem:GLorbits} again we get $w \in (x+y)^{\GL{n,q_0}} \subseteq v^{G_0} + v^{G_0}$. Thus (2) holds.
\end{proof}

\begin{proof}[Proof of Proposition \ref{prop:C5main}]
Suppose that $r-1 \leq c(v) \leq r$. Observe that $\eta(r-1) = \eta(r) = 1$, so for either value of $c(v)$ we have $v^L = \{ u \in V \ | \ c(u) = c(v) \}$, which in turn implies that $v^{G_0} = v^L$. If $c(v) = r$ then $D_v = \F{q}$, and clearly $D_w < D_v$ for any $w \in V^\# \setminus v^{G_0}$. So $w \in v^{G_0} + v^{G_0}$ by part (2) of Lemma \ref{lem:C5diam2}, and thus $V^\# \setminus v^{G_0} \subseteq v^{G_0} + v^{G_0}$. Therefore $\diam{\Gamma} = 2$. Now suppose that $c(v) = r-1$, and let $w \in V^\# \setminus v^{G_0}$. If $c(w) < r-1$ then it follows from part (1) of Corollary \ref{cor:numberoforbits} that $D_w < \lambda D_v = D_{\lambda v}$ for some $\lambda \in \F{q}^\#$. Thus $w \in (\lambda v)^{G_0} + (\lambda v)^{G_0} = v^{G_0} + v^{G_0}$ by Lemma \ref{lem:GLorbits}. If $c(w) = r$ let $x := \sum_{i=1}^{r-1}{\alpha_i v_i}$ and $y := \sum_{i=1}^{r-2}{\beta_i v_i} + \gamma v_r$, where $\{\alpha_1, \ldots, \alpha_{r-1}\}$ is an $\F{q_0}$-basis of $D_v$, $\gamma \in \F{q}^\# \setminus D_v$, and
	\[ \beta_i := \begin{cases}
			\alpha_{i+1} - \alpha_i \ &\text{if } 1 \leq i \leq r-3; \\
			\alpha_1 - \alpha_{r-2} &\text{if } i = r-2.
		      \end{cases} \]
Then $c(x) = c(y) = r-1$ and $c(x+y) = r$, so $x,y \in v^{G_0}$ and $w \in (x+y)^{G_0} \subseteq v^{G_0} + v^{G_0}$. Therefore $V^\# \setminus v^{G_0} \subseteq v^{G_0} + v^{G_0}$, and again we have $\diam{\Gamma} = 2$. This completes the proof of part (1).

If $c(v) = 1$ then we get the special case $v^L = v^{G_0} = (\F{q}V_0)^\#$.  Let $\dist{\0_V,w}{\Gamma}$ denote the distance in $\Gamma$ between the vertices $\0_V$ and $w$; we claim that $\dist{\0_V,w}{\Gamma} = c(w)$ for any $w \in V$. Let $\ell(w) := \dist{\0_V,w}{\Gamma}$. Then $w \in Y$ by Proposition \ref{prop:C5orbits}, where $Y$ is as in (\ref{eq:GLorbits-Y}), so $w$ can be written as a sum of $c(w)$ elements of $(\F{q}V_0)^\#$ and thus $\ell(w) \leq c(w)$. On the other hand $w = \sum_{i=1}^{\ell(w)}{\lambda_i u_i}$, where $\lambda_i \in \F{q}^\#$ and $u_i \in V_0^\#$ for all $i$. Writing each $u_i$ as $u_i = \sum_{j=1}^n{\mu_{i,j}w_j}$ where $\mu_{i,j} \in \F{q_0}$ for all $i,j$, we get $w = \sum_{j=1}^n{\lambda'_j w_j}$ where $\lambda'_j = \sum_{j=1}^{\ell(w)}{\lambda_i\mu_{i,j}}$ for each $j$. Hence $D_w \leq \langle \lambda_1, \ldots, \lambda_{\ell(w)} \rangle_{\F{q_0}}$, so that $c(w) \leq \ell(w)$. Therefore $\ell(w) = c(w)$, as claimed. It follows immediately that $\diam{\Gamma} = \min{n,r}$, and that $\diam{\Gamma} = 2$ if and only if $n = 2$ or $r = 2$. This proves (2).

Suppose that $\diam{\Gamma} = 2$. Then $c(w) \leq 2c(v)$ for any $w \in V^\#$ by part (1) of Lemma \ref{lem:C5diam2}, and in particular $2c(v) \geq \min{n,r}$ since there clearly exists $u \in V^\#$ with $c(u) = \min{n,r}$. Hence $c(v) \leq \frac{1}{2}\min{n,r}$ implies that $\diam{\Gamma} > 2$, and part (3) holds.

Finally, let $a := c(v)$, $S := v^{G_0}$, and $\eta(a)$ as in (\ref{eq:eta}). By Corollary \ref{cor:numberoforbits} we have
	\[ |S| \leq \renewcommand{\arraystretch}{0.8} \left[ \begin{array}{cc} n \\ a \end{array} \right]_{q_0} |\GL{a,q_0}|\left|\F{q}^\# : \F{q_0}^\#\right|s, \]
where $s$ is the largest divisor of $d/n$ with $s \leq \eta(a)$. Hence
	\[ |S|^2 + 1 < q_0^{2an} \left|\F{q}^\# : \F{q_0}^\#\right|^2 s^2. \]
Observe that $s < q_0^{st}$ for all $s \geq 1$, where $t = \frac{9}{17}$ if $q_0 = 2$, and $t = \frac{1}{2}$ if $q_0 \geq 3$. Also, for $q_0 \geq 3$, we have $q_0 - 1 > q_0^{5/8}$, so that $\left|\F{q}^\# : \F{q_0}^\#\right| < q_0^{r-5/8}$. With these bounds we obtain
	\[ |S|^2 + 1 < q_0^{2(a n + r) + k_1(q_0)}, \]
where $k_1(q_0)$ is as defined in (\ref{eq:k1(q_0)}). It is easy to verify that if $a < (r(n-2) - k_1(q_0))/(2n)$ then $2(a n + r) + k_1(q_0) < rn$, so $|S|^2 + 1 < |V|$, and thus $\diam{\Gamma} > 2$ by Lemma \ref{lem:HAgraphs}. This proves part (4).
\end{proof}

\begin{remark}
Some small cases covered by Proposition \ref{prop:C5main} are summarised in Table \ref{table:C5smallcases}. The cases left unresolved by Proposition \ref{prop:C5main} are the following:
	\begin{enumerate}
	\item $5 \leq r \leq n$, $r/2 \leq c(v) \leq r-2$;
	\item $2 = n \leq r-2$, $c(v) = 2$;
	\item $3 \leq n < r$, $\max{ n/2, (r(n-2) - k_1(q_0))/(2n) } \leq c(v) \leq r-2$.
	\end{enumerate}

Let $a := c(v) < r$, $S = v^{G_0}$, and $s$ as in Proposition \ref{prop:C5main} (4). Then $s \geq 1$, $\left|\F{q}^\# : \F{q_0}^\#\right| > q_0^{r-2}$ and
	\[ \renewcommand{\arraystretch}{0.8} \left[ \begin{array}{cc} n \\ a \end{array} \right]_{q_0} |\GL{a,q_0}| > q_0^{2a(n-1)}, \]
so
	\begin{align*}
	|G_0|^2 + 1 &\geq \renewcommand{\arraystretch}{0.8} \left( \left[ \begin{array}{cc} n \\ a \end{array} \right]_{q_0} |\GL{a,q_0}| \left|\F{q}^\# : \F{q_0}^\#\right| s \right)^2 + 1 \\
	&> q_0^{2a(n-1) + 2(r-2)}.
	\end{align*}
It is easy to show that if condition (1) or (2) holds then $2(a(n-1) + r - 2) > rn$, and thus $|G_0|^2 + 1 > |V|$. This, unfortunately, does not lead to any conclusion about $\diam{\Gamma}$.

	\begin{table}
	\begin{center}
	\renewcommand{\arraystretch}{1.25}
	\small
	\begin{tabular}{cccl}
	\hline $r$ & $n$ & $c(v)$ & Conclusion about $\Gamma = \Cay{V,v^{G_0}}$ \\
	\hline\hline
	2 & $\geq 2$ & 1 & $\diam{\Gamma} = 2$ by Proposition \ref{prop:C5main} (2) \\
	  &  & 2 & $\diam{\Gamma} = 2$ by Proposition \ref{prop:C5main} (1) \\
	\hline
	3 & 2 & 1 & $\diam{\Gamma} = 2$ by Proposition \ref{prop:C5main} (2) \\
	  &  & 2 & $\diam{\Gamma} = 2$ by Proposition \ref{prop:C5main} (1) \\
	\hline
	3 & $\geq 3$ & 1 & $\diam{\Gamma} = 3$ by Proposition \ref{prop:C5main} (2) \\
	  &  & 2 & $\diam{\Gamma} = 2$ by Proposition \ref{prop:C5main} (1) \\
	  &  & 3 & $\diam{\Gamma} = 2$ by Proposition \ref{prop:C5main} (2) \\
	\hline
	5 & 2 & 1 & $\diam{\Gamma} = 2$ by Proposition \ref{prop:C5main} (2) \\
	\hline
	5 & 3 & 1 & $\diam{\Gamma} = 3$ by Proposition \ref{prop:C5main} (2) \\
	\hline
	5 & 4 & 1 & $\diam{\Gamma} = 4$ by Proposition \ref{prop:C5main} (2) \\
	  &  & 4 & $\diam{\Gamma} = 2$ by Proposition \ref{prop:C5main} (1) \\
	\hline
	5 & $\geq 5$ & 1 & $\diam{\Gamma} = 5$ by Proposition \ref{prop:C5main} (2) \\
	  &  & 2 & $\diam{\Gamma} > 2$ by Proposition \ref{prop:C5main} (3) \\
	  &  & 4 & $\diam{\Gamma} = 2$ by Proposition \ref{prop:C5main} (1) \\
	  &  & 5 & $\diam{\Gamma} = 2$ by Proposition \ref{prop:C5main} (1) \\
	\hline
	\end{tabular}
	\bigskip
	\caption{$\Gamma$ as in Proposition \ref{prop:C5main} for small values of $r$ and $n$} \label{table:C5smallcases}
	\end{center}
	\end{table}
\end{remark}

\subsubsection{Case $H = \GamSp{n,q}$}

By Theorem \ref{thm:Asch-Sp},
	\[ G_0 = (\GSp{n,q_0} \circ Z_{q-1}) \rtimes \langle \tau \rangle \]
and $L = \GSp{n,q_0} \circ Z_{q-1}$. The main result in this section is parallel to part (4) of Proposition \ref{prop:C5main}.

\begin{proposition} \label{prop:C5bounds-Sp}
Let $\Gamma$ be a graph and $G \leq \Aut{\Gamma}$ such that $G$ satisfies Hypothesis {\rm\ref{hyp:HA}} with $H = \GamSp{n,q}$ and $i=5$. Then $\Gamma$ is connected and $G$-symmetric if and only if $\Gamma \cong \Cay{V,v^{G_0}}$ for some $v \in V^\#$. Moreover, if $s := |\tau| = |G_0 : L|$ and $c(v)$ is as defined in {\rm(\ref{eq:c(v)})}, and if
	\[ t := \begin{cases}
	             9/17 \ &\text{if } q_0 = 2, \\
		     1/2 &\text{if } q_0 > 2
	             \end{cases} \]
then the following hold:
	\begin{enumerate}
	\item If $c(v) < \frac{1}{2}\min{n,r}$ then $\diam{\Gamma} > 2$.
	\item If $3 \leq n \leq r$, $c(v) \geq n/2$ and $r > (n^2 + n + 2st)/(n-2)$, then $\diam{\Gamma} > 2$.
	\end{enumerate}
\end{proposition}

\begin{proof}
Assume that $c(v) < \frac{1}{2}\min{n,r}$. Let $S = v^{G_0}$, and let $\Gamma' = \Cay{V,v^{G'_0}}$, such that $G'$ satisfies Hypothesis {\rm\ref{hyp:HA}} with $H = \GamL{n,q}$ and $i = 5$. Then $\Gamma$ is a subgraph of $\Gamma'$, and hence $\diam{\Gamma} \geq \diam{\Gamma'}$. If $c(v) = 1$ then $\diam{\Gamma'} \geq \min{n,r} > 2$ by part (2) of Proposition \ref{prop:C5main}, and if $c(v) \geq 2$ then $\diam{\Gamma'} > 2$ by part (3) of Proposition \ref{prop:C5main}. In both cases $\diam{\Gamma} > 2$. This proves statement (1).

We now prove statement (2). Observe that for any $\lambda \in \F{q}^\#$ and $g \in \GSp{n,q_0}$, we have $\lambda v^g = v^{\lambda g} \in v^{\GSp{n,q_0}}$ if and only if $\lambda I_n \in Z_{q_0-1}$, the subgroup of scalar matrices in $\GL{n,q_0}$. Hence $v^L = \bigcup_{\lambda \in \F{q}^\#}{\lambda v^{\GSp{n,q_0}}}$ can be written as a disjoint union $v^L = \bigcup_{\lambda \in T}{\lambda v^{\GSp{n,q_0}}}$, where $T$ is a transversal of $\F{q_0}^\#$ in $\F{q}^\#$. Thus
	\[ |v^L| \leq |T||\GSp{n,q_0}| = (q_0^r - 1)|\Sp{n,q_0}| \]
and $|S| \leq s |v^L|$, where $s = |G_0 : L|$. We have
	\[ |\Sp{n,q_0}| = q_0^{n^2/4} \prod_{i=1}^{n/2} \left( q^{2i} - 1 \right) < q_0^{(n^2 + n)/2}. \]
Also, as in the proof of Proposition \ref{prop:C5main} (4), we have $s < q_0^{st}$ for any $s$, where $t = \frac{9}{17}$ if $q_0 = 2$, and $t = \frac{1}{2}$ if $q_0 \geq 3$. Hence
	\[ |S|^2 + 1 < s^2 (q_0^r - 1)^2 q_0^{n^2 + n} < q_0^{n^2 + n + 2r + 2st}. \]
If $r > (n^2 + n + 2st)/(n-2)$ then $rn > n^2 + n + 2r + 2st$, so $|V| > |S|^2 + 1$ and $\diam{\Gamma} > 2$ by Lemma \ref{lem:HAgraphs}. Therefore part (2) holds.
\end{proof}
\subsection{Class $\C_6$}

In this case $\dim{V} = r^t$ where $r$ is a prime different from $p$, $q$ is the smallest power of $p$ such that $q \equiv 1 \pmod{|Z(R)|}$ for some $R$ in Table \ref{table:C6groups}, and
	\[ G_0 = (Z_{q-1} \circ R).T \rtimes \langle \tau \rangle, \]
with $T$ as in Table \ref{table:C6groups}. By Theorems \ref{thm:Asch-GL} and \ref{thm:Asch-Sp}, if $H = \GamL{n,q}$ then $R$ is of type 1 or 2, and if $H = \GamSp{n,q}$ with $q$ odd then $R$ is of type 4.

Proposition \ref{prop:C6bounds} is an extension of \cite[Proposition 3.6]{HAcase}, and is proved somewhat similarly.

\begin{proposition} \label{prop:C6bounds}
Let $V$ and $G_0$ be as above, and let $\Gamma := \Cay{V,S}$ for some $G_0$-orbit $S \subseteq V^\#$.
	\begin{enumerate}
	\item Suppose that $r$ is odd, $q \equiv 1\pmod{r}$, and $R$ is Type 1. If $\diam{\Gamma} = 2$ then $1 \leq t \leq 3$, $r \leq r_0(t)$, and $q \leq q_0(r,t)$, where $r_0(t)$ and $q_0(r,t)$ are given in Table {\rm\ref{table:C6bounds-type1}}.
	\item Suppose that $r = 2$, $t \geq 2$, $q \equiv 1 \pmod{4}$, and $R$ is Type 2. If $\diam{\Gamma} = 2$ then $2 \leq t \leq 6$ and $q \leq q_0(t)$, where $q_0(t)$ is given in Table {\rm\ref{table:C6bounds-type2}}.
	\item Suppose that $r = 2$, $t \geq 2$, $q$ is odd, and $R$ is Type 4. If $\diam{\Gamma} = 2$ then $2 \leq t \leq 7$ and $q \leq q_0(t)$, where $q_0(t)$ is given in Table {\rm\ref{table:C6bounds-type4}}.
	\item Suppose that $r = 2$, $t = 1$, $q$ is odd, and $R$ is Type 2 or 4. Then $\diam{\Gamma} = 2$ for any $S$.
	\end{enumerate}

	\begin{table}[ht]
	\begin{center}
	\renewcommand{\arraystretch}{1.25}
	\begin{tabular}{crrr}
	\hline
	$t$ & $1$ & $2$ & $3$ \\
	\hline\hline
	$r_0(t)$ & $11$ & $3$ & $3$ \\
	\hline
	$q_0(3,t)$ & $186619$ & $73$ & $11$ \\
	$q_0(5,t)$ & $521$ & - & - \\
	$q_0(7,t)$ & $71$ & - & - \\
	$q_0(11,t)$ & $23$ & - & - \\
	\hline
	\end{tabular}
	\bigskip \caption{Bounds for $r$ and $q$ when $R$ is Type 1} \label{table:C6bounds-type1}
	\end{center}
	\end{table}
	
	\begin{table}[ht]
	\begin{center}
	\renewcommand{\arraystretch}{1.25}
	\begin{tabular}{crrrrrr}
	\hline
	$t$ & $2$ & $3$ & $4$ & $5$ & $6$ \\
	\hline\hline
	$q_0(t)$ & $23029$ & $569$ & $73$ & $17$ & $5$ \\
	\hline
	\end{tabular}
	\bigskip \caption{Bounds for $q$ when $R$ is Type 2 and $t \geq 2$} \label{table:C6bounds-type2}
	\end{center}
	\end{table}

	\begin{table}[ht]
	\begin{center}
	\renewcommand{\arraystretch}{1.25}
	\begin{tabular}{crrrrrr}
	\hline
	$t$ & $2$ & $3$ & $4$ & $5$ & $6$ & $7$ \\
	\hline\hline
	$q_0(t)$ & $1913$ & $149$ & $37$ & $11$ & $5$ & $3$ \\
	\hline
	\end{tabular}
	\bigskip \caption{Bounds for $q$ when $R$ is Type 4 and $t \geq 2$} \label{table:C6bounds-type4}
	\end{center}
	\end{table}
\end{proposition}

\begin{proof}
If $q = p^\ell$ and $R$ is Type 1 or 2, then
	\[ |G_0| = \ell (q-1) r^{2t} |\Sp{2t,r}| < \ell (q-1) r^{2t^2 + 3t}. \]
Suppose first that $R$ is Type 1. In this case $r$ is odd and $q = p^\ell \equiv 1 \pmod{r}$, so $\ell \leq r-1$, $q > r$, and
	\[ |G_0|^2 + 1 < \left( (q-1) r^{2t^2 + 3t + 1} \right)^2 + 1 < q^{4t^2 + 6t + 4}. \]
It can be shown that $4t^2 + 6t + 4 < r^t$ for the following cases: $t \geq 5$ and $r \geq 3$, $t = 1$ and $r \geq 17$, $t = 2$ and $r \geq 7$, and $t \in \{3,4\}$ and $r \geq 5$. Thus for all these cases $|G_0|^2 + 1 < |V|$. For all remaining pairs $(r,t)$ define
	\[ \pi(q,r,t) := \left( (r-1)(q-1) r^{2t} |\Sp{2t,r}| \right)^2 + 1 - q^{r^t}. \]
Then $|G_0|^2 + 1 - |V| < \pi(q,r,t)$ and $\pi(q,r,t) < 0$ for $q > \left( (r-1) r^{2t} |\Sp{2t,r}| \right)^{2/\left(r^t - 2\right)}$. Getting the largest prime power $q = p^\ell \equiv 1 \pmod{r}$ less than or equal to this bound, with $\ell \leq r-1$ and $\pi(q,r,t) > 0$, gives the values $q_0(r,t)$ in Table \ref{table:C6bounds-type1}, and for each $t$ we take $r_0(t)$ to be the largest value of $r$ for which there exist such $q$. In particular, $\pi(q,r,t) < 0$ for the following cases: $(r,t) = (13,1)$ and $q > 13$, $(r,t) = (5,2)$ and $q > 7$, $(r,t) = (3,4)$ and $q > 3$; for these cases there is no value of $q$ less than or equal to the given bound that satisfies all the required conditions. This proves part(1).

Now suppose that $R$ is Type 2 with $t \geq 2$. Then $r = 2$ and $q = p^\ell \equiv 1 \pmod{4}$, so $\ell \leq 2$, $q > 4$, and
	\[ |G_0|^2 + 1 < \left( (q-1) 2^{2t^2 + 3t + 1} \right)^2 + 1 < q^{2t^2 + 3t + 3}. \]
We have $2t^2 + 3t + 3 < 2^t$ whenever $t \geq 7$, hence $|G_0|^2 + 1 \leq |V|$ for all such $t$. For $t \in \{1, \ldots, 6\}$ define
	\[ \pi(q,t) := \left( 2 (q-1) 2^{2t} |\Sp{2t,2}| \right)^2 + 1 - q^{2^t}, \]
and observe that $|G_0|^2 + 1 - |V| < \pi(q,t) < 0$ for all $q > \left( 2^{2t+1} |\Sp{2t,2}| \right)^{1/\left(2^{t-1} - 1\right)}$. The values of $q_0(t)$ in Table \ref{table:C6bounds-type2} are the largest prime powers $q = p^\ell \equiv 1 \pmod{4}$ less than or equal to these bounds, with $\ell \leq 2$ and satisfying $\pi(q,t) > 0$. This proves (2).

For (3), suppose that $R$ is Type 4 with $t \geq 2$. Then $r = 2$ and $|Z(R)| = 2$, so $\ell = 1$ and $q = p$. Also $q \geq 3$, so $q^{3/2} > 4$. We have
	\[ |G_0| = (q-1) 2^{2t} \left| \Ominus{2t,2} \right| < (q-1) 2^{2t^2 + t + 2} \]
so
	\[ |G_0|^2 + 1 < \left( (q-1) 2^{2t^2 + t + 2} \right)^2 + 1 < q^2 4^{2t^2 + t + 2} < q^{3t^2 + \frac{3}{2}t + 5}. \]
We have $3t^2 + \frac{3}{2}t + 5 < 2^t$ (and hence $|G_0|^2 + 1 < |V|$) for all $t \geq 8$. For $t \in \{2, \ldots, 7\}$ define
	\[ \pi(q,t) := \left( (q-1) 2^{2t} \left|\Ominus{2t,2}\right| \right)^2 + 1 - q^{2^t}. \]
Then $|G_0|^2 + 1 - |V| < \pi(q,t) < 0$ for all $q > \left( 2^{2t} \left|\Ominus{2t,2}\right| \right)^{1/\left( 2^{t-1} - 1 \right)}$. As in the previous cases we take $q_0(t)$, $2 \leq t \leq 7$, to be the largest prime $q$ less than or equal to these bounds such that $\pi(q,t) > 0$. This yields Table \ref{table:C6bounds-type4} and proves (3).

Statement 4 for the case where $R$ is type 2 is precisely \cite[Proposition 3.6 (2)]{HAcase}. For the case where $R$ is type 4 define the matrices $a, c \in \GL{V}$ by
	\[ a := \left(\begin{matrix} 0 & 1 \\ -1 & 0 \end{matrix}\right) \quad \text{and} \quad
		 c := \left(\begin{matrix} \beta & \gamma \\ \gamma & -\beta \end{matrix}\right), \]
where $\beta,\gamma \in \F{q}$ such that $\beta^2 + \gamma^2 = -1$. Then $\langle a,c \rangle$ is a representation of $R$ in $\GL{2,q}$ (see \cite[pp. 153-154]{KleidLieb}). Since $R$ is irreducible on $V$, any $R$-orbit $v^R$ in $V^\#$ contains a basis $\{v_1,v_2\}$ of $V$, and $v^{G_0}$ contains $\langle v_1 \rangle^\# \cup \langle v_2 \rangle^\#$. Clearly $V^\# \subseteq \langle v_1 \rangle^\# + \langle v_2 \rangle^\#$. Therefore $V \subseteq v^{G_0} + v^{G_0}$, and thus $\diam{\Gamma} = 2$. This proves (4), and completes the proof of the proposition.
\end{proof}

\subsection{Class $\C_7$}

In this case $V = \otimes_{i=1}^t{U_i}$ with $U_i = \F{q}^m$ for all $i$, $m \geq 2$, $t \geq 2$, and $d = m^t$. Assume that $\mathcal{B}$ is a tensor product basis of $V$, with
	\[ \mathcal{B} := \left\{ \otimes_{i=1}^t{u_{i,j}} \ \vline \ 1 \leq j \leq m \right\}. \]
As in the $\C_4$ case, it is not difficult to show that for any $v = \sum_{i=1}^r{\left(\otimes_{j=1}^t{v_{i,j}}\right)} \in V^\#$ we have
	\[ v^\tau = \sum_{i=1}^r{\left(\otimes_{j=1}^t{v_{i,j}^\tau}\right)}, \]
where $\tau$ acts on each $U_i$ with respect to the basis $\{ u_{i,j} \ | \ 1 \leq j \leq m \}$.

\subsubsection{Case $H = \GamL{n,q}$}

By Theorem \ref{thm:Asch-GL}
	\begin{equation} \label{eq:G0-C7}
	G_0 = (\GL{m,q} \wr_\otimes \Sym{t}) \rtimes \langle \tau \rangle.
	\end{equation}

If $t = 2$ then we obtain the examples in Proposition \ref{prop:C4-diam2} with $k = m$. We state this in the next corollary, which is analogous to \cite[Corollary 3.7]{HAcase}.

\begin{corollary} \label{cor:C7-twofactors}
Let $V = \otimes_{i=1}^t{\F{q}^m}$ and let $G_0$ be as in {\rm(\ref{eq:G0-C7})} with $m \geq 2$ and $t = 2$. Then the $G_0$-orbits in $V^\#$ are the sets $Y_s$ for each $s \in \{1, \ldots, m\}$, where $Y_s$ is as defined in {\rm(\ref{eq:Y_s})}. Moreover, for any $G_0$-orbit $S \subseteq V^\#$, the graph $\Cay{V,S}$ has diameter $2$ if and only if $S = Y_s$ for some $s \geq m/2$.
\end{corollary}

\begin{proof}
This follows immediately from Lemma \ref{lem:C4orbits} and Proposition \ref{prop:C4-diam2}.
\end{proof}

Using Lemma \ref{lem:HAgraphs}, we get the following bounds which significantly reduce the cases that remain to be considered. It turns out that these are exactly the same as those in \cite[Proposition 3.8]{HAcase}; we prove them here for subgroups of $\GamL{n,q}$.

\begin{proposition} \label{prop:C7bounds}
Let $\Gamma$ be a graph and let $G \leq \Aut{\Gamma}$, such that $G$ satisfies Hypothesis {\rm\ref{hyp:HA}} with $G_0$ as in {\rm(\ref{eq:G0-C7})}, $m \geq 2$ and $t \geq 3$. Then $\Gamma$ is connected and $G$-symmetric if and only if $\Gamma \cong \Cay{V,v^{G_0}}$ for some $v \in V^\#$. Moreover, if $\diam{\Gamma} = 2$ then either:
	\begin{enumerate}
	\item $m=2$ and $t \in \{3,4,5\}$; or
	\item $t=3$ and $m \in \{3,4,5\}$.
	\end{enumerate}
\end{proposition}

\begin{proof}
Recall that $(\alpha g_1) \otimes g_2 \otimes \cdots \otimes g_t = g_1 \otimes \cdots \otimes (\alpha g_i) \otimes \cdots \otimes g_t$ for all $g_1, \ldots, g_t \in \GL{m,q}$, so that
	\[ |G_0| \leq |\GL{m,q}|^t \, t! \, \ell (q-1)^{-(t-1)}. \]
Now
	\[ |\GL{m,q}| < q^{m(m-1)} q^{m-1} (q-1) = q^{m^2 - 1} (q-1), \]
$s \leq q^{s-1}$ for all $s \geq 2$ and $q \geq 2$, and $\ell < p^\ell = q$ for all $\ell \geq 1$ and $p \geq 2$, so that
	\[ |G_0|^2 + 1 < \left( q^{(m^2 - 1)t} (q-1)^t \right)^2 \left( q^{\frac{1}{2}t(t-1)} \right)^2 q^2 (q-1)^{-2(t-1)} < q^{t^2 + (2m^2 - 3)t + 4}. \]
It can be shown that $t^2 + (2m^2 - 3)t + 4 < m^t$ whenever $t \geq 7$ and $m \geq 2$, and whenever $t \in \{ 3, 4, 5, 6 \}$ and $m > m_0(t)$, where $m_0(t)$ is as given in Table \ref{table:m_0}. Hence $|G_0|^2 + 1 < |V|$ for all such pairs $(m,t)$. Of the remaining pairs we can eliminate $(2,6)$ and $(6,3)$ by considering $\pi(q,m,t) := (t!)^2 q^{2t(m^2 - 1) + 4} - q^{m^t}$; it can be shown that $\pi(q,2,6) < 0$ for all $q \geq 2$ and $\pi(q,6,3) < 0$ for all $q \geq 7$. For $q \in \{2,3,4,5\}$ it can be checked that $36 \, \ell^2 \, |\GL{6,q}|^6 (q-1)^{-4} + 1 < q^{216}$. Therefore $|G_0|^2 + 1 < |V|$ if $(m,t) \in \{ (2,6), (6,3) \}$, which completes the proof.
	\begin{table}[ht]
	\begin{center}
	\begin{tabular}{ccccc}
	\hline
	$t$ & 3 & 4 & 5 & 6 \\
	\hline\hline
	$m_0(t)$ & 6 & 2 & 2 & 2 \\
	\hline
	\end{tabular}
	\bigskip \caption{Values for $m_0(t)$} \label{table:m_0}
	\end{center}
	\end{table}
\end{proof}

\subsubsection{Case $H = \GamSp{n,q}$}

By Theorem \ref{thm:Asch-Sp}, both $q$ and $t$ are odd and
	\begin{equation} \label{eq:G0-C7-Sp}
	G_0 = (\GSp{m,q} \wr_\otimes \Sym{t}) \rtimes \langle \tau \rangle.
	\end{equation}
Hence $q, t \geq 3$.

\begin{proposition} \label{prop:C7bounds-Sp}
Let $\Gamma$ be a graph and $G \leq \Aut{\Gamma}$ such that $G$ satisfies Hypothesis {\rm\ref{hyp:HA}} with $G_0$ as in {\rm(\ref{eq:G0-C7-Sp})}, $m \geq 2$ and $t \geq 3$. Then $\Gamma$ is connected and $G$-symmetric if and only if $\Gamma \cong \Cay{V,v^{G_0}}$ for some $v \in V^\#$. Moreover, if $\diam{\Gamma} = 2$ then either:
	\begin{enumerate}
	\item $m=2$ and $t \in \{3,5\}$; or
	\item $t=3$, $m = 4$, and $q = 9$.
	\end{enumerate}
\end{proposition}

\begin{proof}
In this case $|G_0| \leq |\GSp{m,q}|^t \, t! \, \ell (q-1)^{-(t-1)}$, where
	\[ |\GSp{m,q}| = (q-1) \Sp{m,q} < (q-1) q^{\frac{1}{2}(m^2 + m)}. \]
Also $s \leq k^{s/2}$ for all $k \geq 3$ and $s \geq 2$, so that $\ell \leq q$, $t! \leq q^{\frac{1}{4}(t-1)(t+2)}$, and
	\[ |G_0|^2 + 1 < (q-1)^{2t} q^{t(m^2 + m) + \frac{1}{2}(t-1)(t+2) + 1} (q-1)^{-2(t-1)} < q^{\frac{1}{2}t^2 + \left( m^2 + m + \frac{1}{2}\right)t + 2}. \]
It can be shown that $\frac{1}{2}t^2 + \left( m^2 + m + \frac{1}{2} \right)t + 2 < m^t$ whenever $t \geq 6$ and $m \geq 2$, $t = 3$ and $m \geq 5$, and $t = 5$ and $m \geq 3$. So $|G_0|^2 + 1 < |V|$ for all such pairs $(m,t)$. Let $\pi(q,m,t) := (t!)^2 q^{t(m^2 + m) + 3} - q^{m^t}$. If $(m,t) = (4,3)$ then for all $q \geq 37$ we get
	\[ |G_0|^2 + 1 - |V| < \pi(q,4,3) < 0. \]
For $3 \leq q \leq 31$, $q \neq 9$, we have $36 \, \ell^2 \, (q-1)^2 |\Sp{4,q}|^6 + 1 < q^{64}$. Therefore if $(m,t) = (4,3)$ and $q \neq 9$ then $|G_0|^2 + 1 < |V|$, which completes the proof.
\end{proof}

\subsection{Proof of Theorem \ref{mainthm:HA}}

We now give the proof of Theorem \ref{mainthm:HA}.

\begin{proof}[Proof of Theorem \ref{mainthm:HA}]
The first part follows immediately from Lemma \ref{lem:HAgraphs}, so we only need to show statements (1) -- (3). Assume that $G_0$ does not belong in the Aschbacher class $\C_9$. Line 1 of Table \ref{tab:mainthm-HA} follows from Proposition \ref{prop:C2-diam2}, line 2 from Proposition \ref{prop:C4-diam2}, lines 3 and 4 from Proposition \ref{prop:C5main} (1) and (2), respectively. Line 5 follows from Proposition \ref{prop:C6bounds} (4), line 5 from Corollary \ref{cor:C7-twofactors}, line 7 from Proposition \ref{prop:C8-GUdiam2}, and lines 8 -- 11 from Proposition \ref{prop:C8-GOdiam2}. Line 1 of Table \ref{tab:mainthm-HA-Sp} follows from Proposition \ref{prop:C2-diam2-Sp} (1), line 2 from Proposition \ref{prop:C2-diam2-Sp} (2) and (3), and line 3 from Proposition \ref{prop:C6bounds} (4). Lines 4 -- 6 follow from Proposition \ref{prop:C8-GOdiam2}. This proves statement (1).

Statement (2) follows from the results given in the Restrictions column of Table \ref{tab:mainthm-HA-Sp}. This completes the proof of Theorem \ref{mainthm:HA}.
\end{proof}


\end{document}